\documentclass[12pt,a4paper,intlimits,sumlimits]{amsart}

\usepackage{amsmath, amsthm, mathtools}
\usepackage{amsfonts}
\usepackage[alphabetic]{amsrefs}
\usepackage{graphicx}
\usepackage{framed}
\usepackage{amssymb}
\usepackage{esvect}
\usepackage{xcolor}
\usepackage{eucal}
\usepackage{float}
\usepackage[normalem]{ulem}

\usepackage{hyperref}
\usepackage[english]{babel}

\def \N {\mathbb{N}}
\def \R {\mathbb{R}}
\def \Z {\mathbb{Z}}

\newcommand{\eps}{\varepsilon}
\newcommand{\Ap}[1][]{A_p\, #1}
\newcommand{\Bp}[1][]{B_p\, #1}

\newcommand{\dnorm}[2][]{\left\|#2\right\|_{#1}^*}
\newcommand{\dualp}[3][]{\left(#2,#3\right)_{#1}}

\newcommand{\norm}[2][]{\left\|#2\right\|_{#1}}
\renewcommand{\o}{\text{o}}
\newcommand{\PS}[1]{$(\text{PS})_{#1}$}
\newcommand{\pnorm}[2][]{\if #1'' \left|#2\right|_p \else \left|#2\right|_{#1} \fi}

\newcommand{\set}[1]{\left\{#1\right\}}

\newcommand\redsout{\bgroup\markoverwith{\textcolor{red}{\rule[0.5ex]{2pt}{0.4pt}}}\ULon}

\theoremstyle{definition}
\newtheorem{definition}{Definition}[section]

\theoremstyle{plain}
\newtheorem{theorem}[definition]{Theorem}
\newtheorem{proposition}[definition]{Proposition}
\newtheorem{lemma}[definition]{Lemma}
\newtheorem{corollary}[definition]{Corollary}

\numberwithin{equation}{section}

\usepackage{geometry}
\renewcommand{\leq}{\leqslant}
\renewcommand{\le}{\leqslant}
\renewcommand{\geq}{\geqslant}
\renewcommand{\ge}{\geqslant}

\allowdisplaybreaks

 \title[Nonlinear superposition operators of mixed fractional order]{An existence theory \\ for
nonlinear superposition operators \\ of mixed fractional order}
 
\author{Serena Dipierro}
\address{Serena Dipierro: Department of Mathematics and Statistics, The University of Western Australia, 35 Stirling Highway, Crawley, Perth, WA 6009, Australia}
\email{serena.dipierro@uwa.edu.au}

\author{Kanishka Perera}
\address{Kanishka Perera: Department of Mathematical Sciences, Florida Institute of Technology, 150 W University Blvd, Melbourne, FL 32901-6975, USA}
\email{kperera@fit.edu}

\author{Caterina Sportelli}
\address{Caterina Sportelli: Department of Mathematics and Statistics, The University of Western Australia, 35 Stirling Highway, Crawley, Perth, WA 6009, Australia}
\email{caterina.sportelli@uwa.edu.au}

\author{Enrico Valdinoci}
\address{Enrico Valdinoci: Department of Mathematics and Statistics, The University of Western Australia, 35 Stirling Highway, Crawley, Perth, WA 6009, Australia}
\email{enrico.valdinoci@uwa.edu.au}
 
 \date{}

\makeindex

\begin{document}
 \maketitle

\begin{abstract}
We establish the existence of multiple solutions for a nonlinear problem of critical type.

The problem considered is fractional in nature, since it is obtained by the superposition of~$(s,p)$-fractional Laplacians of different orders.

The results obtained are new even in the case of the sum of two different fractional~$p$-Laplacians, or the sum of a fractional~$p$-Laplacian and a classical~$p$-Laplacian, but our framework is general enough to address also the sum of finitely, or even infinitely many, operators.

In fact, we can also consider the superposition of a continuum of operators, modulated by a general signed measure on the fractional exponents. When this measure is not positive, the contributions of the individual operators to the whole superposition operator is allowed to change sign. In this situation, our structural assumption is that the positive measure on the higher fractional exponents dominates the rest of the signed measure.
\end{abstract}

\tableofcontents{}
 
\section{Introduction}

\subsection{Nonlinear nonlocal operators and existence theory}\label{AB1}

The goal of this article is to construct solutions for an operator obtained through the superposition
of possibly nonlinear, fractional operators. The results obtained will be very general, but they are also new in a number of specific interesting cases which will be obtained as a simple byproduct of our comprehensive approach.

We consider two finite (Borel) measures~$\mu^+$ and~$\mu^-$ on the interval of fractional exponents~$[0,1]$
satisfying, for some~$\overline s\in ({{ 0 }}, 1]$ and~$\gamma\ge0$, that
\begin{equation}\label{mu00}
{\mu^+}([\overline s, 1])>0,
\end{equation}
\begin{equation}\label{mu3}
{\mu^-}_{\big|_{[\overline s, 1]}}=0
\end{equation}
and
\begin{equation}\label{mu2}
\mu^-\big([{{ 0 }}, \overline s]\big)\le\gamma\,
\mu^+\big([\overline s, 1]\big).
\end{equation}
We consider the signed measure
$$ \mu:=\mu^+-\mu^-$$
and the nonlinear fractional operator
\[ A_{\mu,p}u:=\int_{[0, 1]} (-\Delta)_p^{s} u \, d\mu(s).\]
Notice that the above integration may occur with ``variable signs'', since~$\mu$ is a signed measure
with possibly negative components induced by the measure~$\mu^-$: however, conditions~\eqref{mu00} and~\eqref{mu3} say that these negative components do not affect the higher values of fractional exponents and, in fact, by condition~\eqref{mu2},
the measure of these higher fractional exponents suitably dominate the negative contribution.

This setting has been introduced in~\cite{CATERINA} for the case of linear fractional operators (corresponding to~$p=2$). Here, we can also allow a nonlinear nonlocal structure of the operator, hence we assume that~$p\in(1,+\infty)$.

In our setting, when~$s\in(0,1)$ the fractional~$p$-Laplacian~$(- \Delta)_p^s$ is the nonlinear nonlocal operator defined on smooth functions by
\[
(- \Delta)_p^s\, u(x) := 2c_{N,s,p} \,\lim_{\eps \searrow 0} \int_{\R^N \setminus B_\eps(x)} \frac{|u(x) - u(y)|^{p-2}\, (u(x) - u(y))}{|x - y|^{N+sp}}\, dy.
\]
See e.g.~\cite{MR3593528, MR3861716} and the references therein for more information on the fractional~$p$-Laplace operator.
The exact value of the positive normalizing constant~$c_{N,s,p}$ is not important\footnote{However, one can take
$$ c_{N,s,p}:=\frac{s\,2^{2s-1}\,\Gamma\left(\frac{ps+p+N-2}{2}\right)}{\pi^{N/2}\,\Gamma(1-s)},$$
see e.g.~\cite[page~130]{MR3473114} and the references therein.} for us, except for allowing
a consistent setting for the limit cases, namely that
\begin{eqnarray*}&&\lim_{s\searrow0}(- \Delta)_p^s\, u=(- \Delta)_p^0\, u:=u\\{\mbox{and}}\qquad&&
\lim_{s\nearrow1}(- \Delta)_p^s\, u=(- \Delta)_p^1\, u:=-\Delta_pu=-{\operatorname{div}}(|\nabla u|^{p-2}\nabla u).\end{eqnarray*}\medskip

To develop our analysis, we consider a bounded open set~$\Omega\subset\R^N$
and rely on the spectral theory for the operator~$A_{\mu,p}$. For this, one looks
at the values of~$\lambda$ allowing for nontrivial solutions of the equation
\begin{equation}\label{EUGE} \begin{cases}A_{\mu,p}u=\lambda |u|^{p-2}u &{\mbox{ in }}\Omega,\\u=0&{\mbox{ in }}\R^N\setminus\Omega.\end{cases}\end{equation}
The cohomological index theory by Fadell and Rabinowitz (put forth in~\cite{MR0478189}
and whose essential ingredients will also be recalled in the forthcoming Section~\ref{AB2})
addresses the Dirichlet eigenvalue theory related to equation~\eqref{EUGE},
provides the existence of a sequence~$\lambda_l$ of eigenvalues, with~$0<\lambda_1\le\lambda_2\le\dots$,
and will constitute an important building block for our existence theory.
\medskip

Given the superposition nature of the operator, there is another fractional exponent which may play a critical role.
To appreciate this, we observe that, by assumption~\eqref{mu00}, 
there exists~$s_\sharp\in [\overline s, 1]$ such that
\begin{equation}\label{scritico}
\mu^+ ([s_\sharp, 1])> 0.
\end{equation}
We will see below that the exponent~$s_\sharp$ plays also the role of a critical exponent
(and we remark that, while some arbitrariness is allowed in the choice of~$s_\sharp$ here above,
the results obtained will be stronger if one picks~$s_\sharp$ ``as large as possible''
but still fulfilling~\eqref{scritico}).

Specifically, we will look at the equation
\begin{equation} \label{mainp}
\left\{\begin{aligned}
A_{\mu,p}u& = \lambda\, |u|^{p-2}\, u + |u|^{p_{s_\sharp}^* - 2}\, u && \text{in } \Omega,\\
u & = 0 && \text{in } \R^N \setminus \Omega.
\end{aligned}\right.
\end{equation}
Notice that if a nontrivial solution~$u$ exists, then~$-u$ is a nontrivial solutions as well (hence, nontrivial solutions
go automatically in pairs).

The exponent~$p_{s_\sharp}^*$ is the fractional critical exponent related to 
the fractional parameter~$s_\sharp$ for which~\eqref{scritico} holds true, namely
$$ p_{s_\sharp}^*:=\frac{Np}{N-s_\sharp p}\,,$$
under the condition\footnote{{F}rom now on, this condition will always be assumed, to allow us to write critical exponents
in our setting; this condition can also be relaxed by taking~$N>s_\star p$, being~$s_\star$ the supremum of the support of~$\mu$.}
that~$N>p$.
\medskip

We also introduce the notion of Sobolev constant suitable for our framework. Namely,
we set
$$ [u]_{s,p}:=\begin{dcases}
\|u\|_{L^p(\R^N)} & {\mbox{ if }}s=0,\\
\displaystyle\left( c_{N,s,p}\iint_{\R^{2N}}\frac{|u(x)-u(y)|^p}{|x-y|^{N+sp}}\,dx\,dy\right)^{1/p}
& {\mbox{ if }}s\in(0,1),\\
\|\nabla u\|_{L^p(\R^N)} & {\mbox{ if }}s=1.
\end{dcases}$$
Thanks to the normalizing constant~$c_{N,s,p}$ we have that
$$ \lim_{s\searrow0}[u]_{s,p}=[u]_{0,p}\qquad{\mbox{and}}\qquad
\lim_{s\nearrow1}[u]_{s,p}=[u]_{1,p}.$$
The Sobolev constant that we consider is thus defined as
\begin{equation}\label{DESOCO} {\mathcal S(p)}:=\inf \int_{[0,1]}[u]_{s,p}^p \,d\mu^+(s),\end{equation}
where the infimum is taken over all the measurable functions with~$\|u\|_{L^{p^*_{s_\sharp}}(\R^N)}=1$.
\medskip

With this, our main result goes as follows:

\begin{theorem} \label{mainp1} 
Let~$\mu=\mu^+-\mu^-$ with~$\mu^+$ and~$\mu^-$ satisfying~\eqref{mu00}, \eqref{mu3} and~\eqref{mu2}. Let~$s_\sharp$ be the exponent given by~\eqref{scritico}.

Assume that
\begin{equation} \label{LLP1}
\lambda_l - \frac{\mathcal S(p)}{|\Omega|^{(s_\sharp p)/N}} < \lambda < \lambda_l = \dots = \lambda_{l+m-1}
\end{equation}
for some~$l, m \ge 1$.

Then, there exists~$\gamma_0>0$ depending only on~$N$, $\Omega$, $p$, $s_\sharp$, $\mu$, $\lambda$ and~$l$ such that
if~$\gamma\in[0,\gamma_0]$,
problem~\eqref{mainp} has~$m$ distinct pairs of nontrivial solutions~$\pm u^\lambda_1,\dots,\pm u^\lambda_m$.
\end{theorem}

As an immediate consequence of this result, taking~$m:=1$ here above, we obtain:

\begin{corollary} 
Let~$\mu=\mu^+-\mu^-$ with~$\mu^+$ and~$\mu^-$ satisfying~\eqref{mu00}, \eqref{mu3} and~\eqref{mu2}. Let~$s_\sharp$ be the exponent given by~\eqref{scritico}.

Assume that, for some~$l\ge 1$,
\begin{equation*}
\lambda_l - \frac{\mathcal S(p)}{|\Omega|^{(s_\sharp p)/N}} < \lambda < \lambda_l .
\end{equation*}

Then, there exists~$\gamma_0>0$ depending only on~$N$, $\Omega$, $p$, $s_\sharp$, $\mu$, $\lambda$ and~$l$ such that
if~$\gamma\in[0,\gamma_0]$,
problem~\eqref{mainp} has a nontrivial solution
(and therefore a pair of nontrivial solutions, one equal to the other up to a minus sign).
\end{corollary}

Taking~$\mu$ as the Dirac measure at~$1$, Theorem~\ref{mainp1} recovers the
multiplicity result stated in~\cite[Theorem 1.1]{MR3469053}, see
the forthcoming Corollary~\ref{C1}.

Also, taking~$\mu$ as the Dirac measure at some~$s\in(0,1)$, 
Theorem~\ref{mainp1} recovers~\cite[Theorem 1.1]{PSY}
see Corollary~\ref{C2}.

Theorem~\ref{mainp1} is new and very general, in view of the structure of the superposition operator~$A_{\mu,p}$.
Indeed, to the best of our knowledge, several special cases of interest of Theorem~\ref{mainp1} happen to be new.
In particular:
\begin{itemize}
\item Let $s\in [0, 1)$ and set~$\mu:= \delta_1 +\delta_s$,
where~$\delta_1$ and~$\delta_s$ denote the Dirac measures centered at the points~$1$ and~$s$, respectively. With this choice, the operator~$A_{\mu,p}$
boils down to~$-\Delta_p  + (-\Delta)_p^s$. 

We will deal with this case in Corollary~\ref{C3}.
As far as we know, this result is new, even when~$p=2$ (for this case, we refer the interested reader to~\cite{arXiv.2209.07502} for some existence results).\medskip


\item Let~$s\in[0,1)$ and~$\mu:= \delta_1 -\alpha\delta_s$,
where~$\alpha$ is a small positive constant. 
This choice corresponds to the operator~$-\Delta_p-\alpha(-\Delta)^s_p$.
Notice that with this choice the measure~$\mu$ changes sign and the second term in the 
operator has the ``wrong'' sign. 

This case is new and will be treated in Corollary~\ref{C5}.
\medskip

\item Let~$1\ge s_0 > s_1> s_2 >\dots \ge 0$ and
$$\mu:= \sum_{k=0}^{+\infty} c_k \delta_{s_k},\qquad{\mbox{where }}\,
\sum_{k=0}^{+\infty} c_k\in(0,+\infty),$$
with 
\begin{itemize}
\item[(i)] either~$c_0>0$ and~$c_k\ge0$ for all~$k\ge 1$,
\item[(ii)] or
\begin{eqnarray*}& &
c_k>0\ \text{ for all } k\in\{0,\dots, \overline k\} \text{ and } \sum_{k=\overline k +1}^{+\infty} c_k \le \gamma \sum_{k=0}^{\overline k} c_k,\\
&&{\mbox{for some~$\overline k\in\N$ and~$\gamma\ge0$.}}\end{eqnarray*}
\end{itemize}
These two cases of convergent series of Dirac measures are new and will be deal with
in Corollaries~\ref{serie1} and~\ref{serie2}.
\medskip

\item Let~$f$ be a measurable and non identically zero function and consider
a continuous superposition of fractional~$p$-Laplacian of the form
$$ \int_0^1 f(s) (-\Delta)_p^s \, u \, ds.$$
This case is also new in the literature and will be treated in Corollary~\ref{function}.
\end{itemize}

We will discuss in detail all these cases in Section \ref{sec-app}.
\medskip

An interesting feature of the operators considered here is that not only we can deal with nonlinear operators
and infinitely many (possibly uncountably many) fractional operators at the same time, but also
that some of these operators may have the ``wrong sign'' (provided that there is a ``dominant part'' given by the operators with higher fractional order). 

Indeed, we think that the possibility of dealing with complicated operators having some components
with the ``wrong sign'' is particularly intriguing in view of applications in biology and population dynamics.
For example, in light of the L\'evy flight foraging hypothesis, it is often appropriate to model animal dispersal as
a superposition of (possibly fractional) operators of different order, corresponding to different individuals of a given population
adopting different foraging strategies (see e.g.~\cite{MR4249816}). In this sense, the possibility of including also
operators with the ``wrong sign'' provides a simple way to model individuals which, rather than diffusing to search for food,
tend to concentrate together (e.g., for social or mating reason), exhibiting patterns induced by a retrograde fractional
heat equation.

\subsection{An abstract formulation}\label{AB2}

It will now be convenient to cast the framework described in Section~\ref{AB1} into an ``abstract'' setting.
The advantage of this procedure is to develop all the methodology at the level of a suitable functional analysis,
relying on general topological methods, which will then provide the proof of the main result in Theorem~\ref{mainp1}
as a simple byproduct.

To this end, we consider 
\begin{equation}\label{UNICO}
{\mbox{a uniformly convex Banach space~$W$,}}\end{equation} with its norm~$ \norm{\, \cdot\,}$.
The dual space will be denoted by~$W^*$, endowed with the dual norm~$\dnorm{\, \cdot\,}$. The duality pairing will be denoted by~$\dualp{\cdot}{\cdot}$. 

In this setting, one says that~$f \in C(W,W^*)$ is a potential operator if there is a functional~$F \in C^1(W,\R)$, called a potential for~$f$, such that~$F' = f$. 

We consider 
\begin{equation}\label{UNICO2}
{\mbox{two potential operators~$\Ap, \Bp \in C(W,W^*)$}}\end{equation} satisfying the following assumptions:
\begin{enumerate}
\item[$(A_1)$] the operator~$\Ap$ is~$(p - 1)$-homogeneous and odd for some~$p \in (1,+\infty)$, i.e.,
for all~$u \in W$ and~$t \in\R$, we have that
$$\Ap[(tu)] = |t|^{p-2}\, t\, \Ap[u],$$
\item[$(A_2)$] for all~$u, v \in W$, we have that
$$\dualp{\Ap[u]}{v} \le \norm{u}^{p-1} \norm{v}$$
and equality holds if and only if~$\alpha u = \beta v$ for some~$\alpha, \beta \ge 0$, with~$\alpha^2+\beta^2\ne0$,
\item[$(B_1)$] the operator~$\Bp$ is~$(p - 1)$-homogeneous and odd, i.e., for all~$u \in W$ and~$t \in\R$,
$$\Bp[(tu)] = |t|^{p-2}\, t\, \Bp[u],$$ 
\item[$(B_2)$] for all~$u \in W \setminus \set{0}$ we have that
$$\dualp{\Bp[u]}{u} > 0,$$  and for all~$u, v \in W$ we have that
$$\dualp{\Bp[u]}{v} \le \dualp{\Bp[u]}{u}^{(p-1)/p} \dualp{\Bp[v]}{v}^{1/p},$$
\item[$(B_3)$] $\Bp$ is a compact operator.
\end{enumerate}

We will develop the theory in this abstract framework and then we will reduce it to the case of interest in
Section~\ref{AB1} by taking
\begin{equation}\label{2wePAL:a}A_pu:=\int_{[0, 1]} (-\Delta)_p^{s} u \, d\mu^+(s)\qquad{\mbox{ and }}\qquad B_p u:= |u|^{p- 2}\, u .\end{equation}

We define
\begin{equation}\label{SETT1}
I_p(u) := \frac{1}{p} \dualp{\Ap[u]}{u}\qquad{\mbox{and}} \qquad J_p(u) := \frac{1}{p} \dualp{\Bp[u]}{u}.
\end{equation}
The interest of these objects is that \begin{equation}\label{ARETRH} \begin{split}&{\mbox{$I_p$ and~$J_p$
are the potentials of~$\Ap$ and~$\Bp$,}}\\&{\mbox{satisfying~$I_p(0) = 0$ and~$J_p(0) = 0$, respectively,}}\end{split}\end{equation}
see~\cite[Proposition 3.1]{MR4293883}. 

We also observe that, as a consequence of assumption~$(A_2)$, for all~$u \in W$ we have that
$$\dualp{\Ap[u]}{u} = \norm{u}^p$$
and therefore
\begin{equation}\label{SETT2} I_p(u) = \frac{1}{p} \norm{u}^p.\end{equation}
\medskip

The reason for introducing hypotheses~$(A_1)$, $(A_2)$, $(B_1)$, $(B_2)$ and~$(B_3)$ is that, under the above assumptions, it is known that the nonlinear eigenvalue problem
\begin{equation}\label{2.888}
\Ap[u] = \lambda \Bp[u] \quad \text{in } W^*
\end{equation}
possesses an unbounded sequence of eigenvalues~$\lambda_1\le\lambda_2\le\dots$ and that~$\lambda_1>0$, see~\cite{MR1998432}, and also~\cite[Theorem 4.6]{MR2640827} and~\cite[Theorem 1.3]{MR4293883} for full details on this topic.\medskip

Now we consider 
\begin{equation}\label{2.888-0}
{\mbox{a potential operator~$f \in C(W,W^*)$.}}\end{equation} 
We assume that
\begin{equation}\label{2.888-01}
{\mbox{$f$ is odd (i.e. $f(-u)=-f(u)$)}}\end{equation} and 
that the potential~$F$ of~$f$ is normalized such that~$F(0) = 0$.
\medskip

We also consider 
\begin{equation}\label{lppotop}
{\mbox{a potential operator~$L_p\in C(W, W^*)$}}\end{equation}
such that
\begin{enumerate}
\item[$(L_1)$] the operator~$L_p$ is~$(p - 1)$-homogeneous and odd for some~$p \in (1,+\infty)$, i.e.,
for all~$u \in W$ and~$t \in\R$, we have that
$$L_p(tu) = |t|^{p-2}\, t\, L_p(u).$$
\end{enumerate}
Additionally, we define
$$ N_p(u):=\frac1p(L_pu,u).$$
In this setting, $N_p$ is the potential of~$L_p$.

\medskip

Given~$\lambda > 0$, our goal is now to study the equation
\begin{equation} \label{2.10}
\Ap[u] = \lambda \Bp[u] + L_p u + f(u) \quad \text{in } W^*.
\end{equation}

This abstract formulation will then be reduced to the concrete case showcased in Section~\ref{AB1} through the choice
$$L_pu:=\int_{[0, 1]} (-\Delta)_p^{s} u \, d\mu^-(s)$$ and~$f(u):= |u|^{p_{s_\sharp}^* - 2}\, u$ (together with the setting in~\eqref{2wePAL:a}).

To study~\eqref{2.10} in its more general formulation, it is convenient to define, for all~$ u \in W$,
 \begin{equation}\label{EDE}
    E(u) := I_p(u) - N_p(u) - \lambda J_p(u) - F(u)   . \end{equation}
The convenience of this definition is that, in our setting, $E$ will play the role of the variational functional
associated with equation~\eqref{2.10}.\medskip
    
We suppose that the following structural assumptions are satisfied:
\begin{enumerate}
\item[$(\mathcal F_1)$] $F(u) = \o(\norm{u}^p)$ as~$u \to 0$,
\item[$(\mathcal F_2)$] there exist~$\beta \in( 0,+\infty)$ and~$q \in( p,+\infty)$ such that, for all~$u \in W$,
$$F(u) \ge \dfrac{\beta}{q}\, \big(p\, J_p(u)\big)^{q/p},$$
\item[$(\mathcal F_3)$] there exists~$c^* \in( 0,+\infty)$ such that the functional~$E$ satisfies the Palais-Smale
condition~\PS{c} for all~$c \in (0,c^*)$,
\item[$(\mathcal N_1)$] there exists~$\eta\in (0, 1)$ such that
$$0\le N_p(u)\le \eta\, I_p(u).$$
\end{enumerate}\medskip

Our main result in this framework\footnote{As a technical observation, we stress that the eigenvalues that appear in the forthcoming Theorem~\ref{Theorem 2.9} are defined using the cohomological index. It is not clear that this theorem is true for the standard sequence of eigenvalues defined using the genus since the proof of Theorem~\ref{Theorem 2.6} is based on the piercing property of the cohomological index and the genus does not have this property.}
goes as follows:

\begin{theorem} \label{Theorem 2.9}
Suppose that
\begin{equation} \label{2.11}
\lambda_l - \beta^{p/q} \left(\frac{pqc^*}{q - p}\right)^{1 - p/q} < \lambda < \lambda_l = \dots = \lambda_{l+m-1}
\end{equation}
for some~$l, m \ge 1$.

Then, there exists~$\eta_0>0$, depending only on~$N$, $\Omega$, $p$, $\mu$,
$\lambda$ and~$l$, such that if~$\eta\in[0,\eta_0]$, then
equation~\eqref{2.10} has~$m$ distinct pairs of nontrivial solutions~$\pm u^\lambda_1,\dots,\pm u^\lambda_m$.
\end{theorem}

An immediate consequence of this result occurs by taking~$m:=1$ here above. In this way, we obtain:

\begin{corollary}\label{poaskdjvpoelfP}
Suppose that, for some~$l \ge 1$,
\[
\lambda_l - \beta^{p/q} \left(\frac{pqc^*}{q - p}\right)^{1 - p/q} < \lambda < \lambda_l
.\]

Then, there exists~$\eta_0>0$, depending only on~$N$, $\Omega$, $p$, $\mu$,
$\lambda$ and~$l$, such that if~$\eta\in[0,\eta_0]$, then
equation~\eqref{2.10} has a nontrivial solution
(and therefore a pair of nontrivial solutions, one equal to the other up to a minus sign).
\end{corollary}
\medskip

The rest of this paper is organized as follows. In Section~\ref{ANasqdf}
we deal with the abstract formulation and prove Theorem~\ref{Theorem 2.9}.

Then, we focus on the proof of Theorem~\ref{mainp1}. For this, 
in Section~\ref{SOBOLEV} we present some uniform inequalities of Sobolev type, of independent interest,
and in Section~\ref{BaUPS}
we introduce a suitable functional setting to deal with the problem under consideration.

The proof of Theorem~\ref{mainp1} is completed in Section~\ref{BaUPS-2}, by checking that
the hypotheses of the abstract result in Theorem~\ref{Theorem 2.9} are fulfilled. 

In Section~\ref{sec-app} we ``specialize" the superposition operator~$A_{\mu,p}$ and we employ Theorem~\ref{mainp1} to provide new multiplicity results for many cases of interest.

\section{Proof of Theorem~\ref{Theorem 2.9}}\label{ANasqdf}

Given~$r > 0$, we consider the sphere of radius~$r$ in~$W$, namely
$$S_r := \set{u \in W : \norm{u} = r}.$$
We also use the short notation~$S:=S_1$.
In view of~\eqref{SETT1} and~\eqref{SETT2}, we have that
$$ S = \set{u \in W \,:\, I_p(u) = \frac{1}{p}}.$$

For every~$u\in S$, we let
\begin{equation} \label{2.8}
\Psi(u) := \frac{1}{p\, J_p(u)} .
\end{equation}

It is known that critical points of~$\Psi$ in~$S$ are related to solutions of the eigenvalue problem in~\eqref{2.888},
see~\cite[Theorem 4.6]{MR2640827} and~\cite[Theorem 1.3]{MR4293883}, and, in particular, that
\begin{equation} \label{2.9}
\lambda_1 = \inf_{u \in S}\, \Psi(u) > 0.
\end{equation}

To prove Theorem~\ref{Theorem 2.9}, 
we will rely on an abstract critical point theory, detailed in~\cite{MR3469053}, based on the~$\Z_2$-cohomological index introduced by Fadell and Rabinowitz in~\cite{MR0478189} and denoted here by~$i(\cdot)$.



To this end, we say that a subset~$A$ of~$W$ is symmetric
if~$x\in A$ if and only if~$-x\in A$. 

When~$E$ is an even functional satisfying~$(\mathcal F_3)$, the following result holds true:

\begin{theorem}[{\cite[Theorem 2.2]{MR3469053}}] \label{Theorem 2.6}
Let~$A_0$ and~$B_0$ be symmetric subsets of the unit sphere~$S $ such that~$A_0$ is compact and~$B_0$ is closed.

Assume that, for some integers~$l, m \ge 1$, we have that
\[
i(A_0) \ge l + m - 1 \qquad{\mbox{and}}\qquad i(S \setminus B_0) \le l - 1.
\]

Let~$r\in(0,+\infty)$ and~$R\in(r,+\infty)$. We define
\begin{equation}\label{LAUa}\begin{split}
&A := \set{Ru : u \in A_0},\\
&B: = \set{ru : u \in B_0}\\
{\mbox{and }}\quad&X := \set{tu : u \in A,\, 0 \le t \le 1}.\end{split}\end{equation}

We assume that
\begin{equation} \label{2.7}
\sup_{u \in A}\, E(u) \le 0 < \inf_{u \in B}\, E(u) \qquad{\mbox{and}}\qquad \sup_{u \in X}\, E(u) < c^*.
\end{equation}



Then, $E$ has~$m$ distinct pairs of critical points.
\end{theorem}

Given~$a \in \R$, we consider the sublevel and superlevel sets of~$\Psi$, as defined in~\eqref{2.8}, namely we set
$$\Psi^a := \set{u \in S : \Psi(u) \le a}\qquad{\mbox{and}}\qquad\Psi_a = \set{u \in S : \Psi(u) \ge a}.$$
With this notation, we have:

\begin{theorem}[{\cite[Theorem 4.6]{MR2640827}} and {\cite[Theorem 1.3]{MR4293883}}] \label{Theorem 2.7}

Assume that~$l\ge2$ and that~$\lambda_{l-1} < \lambda_l$. 

Then,
\[
i(\Psi^{\lambda_{l-1}}) = i(S \setminus \Psi_{\lambda_l}) = l - 1
\]
and~$\Psi^{\lambda_{l-1}}$ has a compact symmetric subset of index~$l - 1$.
\end{theorem}

With these preliminary results, we can now complete the proof of Theorem~\ref{Theorem 2.9}, by arguing as follows:

\begin{proof}[Proof of Theorem~\ref{Theorem 2.9}]
Regarding condition~\eqref{2.11},
we can take~$l$ as small as possible and~$m$ as large as possible  satisfying~$\lambda_l = \dots = \lambda_{l+m-1}$.

Therefore, we may assume that 
\begin{equation}\label{REDU1}
\lambda_{l+m-1} < \lambda_{l+m}\end{equation}
and (unless~$l=1$) that
\begin{equation}\label{REDU2}
\lambda_{l-1} < \lambda_l.
\end{equation} 

In view of~\eqref{REDU1}, we can use
Theorem~\ref{Theorem 2.7} (with~$l$ replaced here by~$l+m$, which is greater than or equal to~$2$) and deduce that~$\Psi^{\lambda_{l+m-1}}$ possesses a compact symmetric subset~$C$ of index~$l + m - 1$. 

We now take~$A_0 := C$ and~$B_0 := \Psi_{\lambda_l}$ and we observe that, by construction,
$$ i(A_0) = l + m - 1.$$

We claim that
\begin{equation}\label{INPRO}
i(S \setminus B_0) = l - 1.
\end{equation}
To prove this, we distinguish two cases, either~$l=1$ or~$l\ge2$.

If~$l = 1$, then we can use~\eqref{2.9} and deduce that~$B_0 = S$, which gives that~$i(S \setminus B_0) = 0$, thus establishing~\eqref{INPRO} in this case.

If instead~$l \ge 2$, we can use~\eqref{REDU2}, which in turn allows us to Theorem~\ref{Theorem 2.7} and infer that~$
i(S \setminus \Psi_{\lambda_l}) = l - 1$, which proves~\eqref{INPRO} in this case as well.

Now we recall~\eqref{2.8} and we observe that,
for all~$u \in S$ and~$t \ge 0$,
\begin{equation} \label{2.12}
E(tu) = t^p\, \big(I_p(u) -N_p(u) - \lambda J_p(u)\big) - F(tu) = \frac{t^p}{p} \left(1 -pN_p(u) - \frac{\lambda}{\Psi(u)}\right) - F(tu).
\end{equation}

We also pick~$r\in(0,+\infty)$ and~$R\in(r,+\infty)$ and we define~$A$, $B$, and~$X$ 
as in~\eqref{LAUa}. Our goal is to find~$r$ and~$R$ such that condition~\eqref{2.7} is satisfied. To this end,
we recall~\eqref{2.8} and~$(\mathcal F_2)$, to see that, for all~$u \in A_0 \subset \Psi^{\lambda_{l+m-1}} = \Psi^{\lambda_l}$, 
\[
F(Ru) \ge \dfrac{\beta R^q}{q}\, \big(p\, J_p(u)\big)^{q/p} = \dfrac{\beta R^q}{q\, \Psi^{q/p}(u)} \ge \dfrac{\beta R^q}{q\, \lambda_l^{q/p}}.
\]
This, together with~\eqref{2.12}, gives that
\begin{equation} \label{2.13}
E(Ru) \le \frac{R^p}{p} \left(1 - \frac{\lambda}{\lambda_l}\right) - \dfrac{\beta R^q}{q\, \lambda_l^{q/p}}.
\end{equation}

Regarding~$(\mathcal F_2)$,
we stress that~$\beta > 0$ and~$q > p$. As a result, the first inequality in~\eqref{2.7} holds true 
as a consequence of~\eqref{2.13} as long as~$R$ is sufficiently large.

Now, to complete the check of the validity of~\eqref{2.7},
we recall that~$F(ru) = \o(r^p)$, thanks to~$(\mathcal F_1)$.

Hence, we exploit~\eqref{2.12},  $(\mathcal F_1)$
and~$(\mathcal N_1)$ to see that, for all~$u \in B_0 = \Psi_{\lambda_l}$, 
\[
E(ru) \ge \frac{r^p}{p} \left(1-\eta - \frac{\lambda}{\lambda_l} + \o(1)\right) \text{ as } r \to 0
.\]
Since~$\lambda < \lambda_l$, it follows from this that the second inequality in~\eqref{2.7} also holds if~$r$ and~$\eta$ are sufficiently small. 

Furthermore, we employ~\eqref{2.13} to see that,
for all~$u \in A$ and~$0 \le t \le 1$,
\begin{equation} \label{2.14}
E(tu) \le \frac{t^p R^p}{p} \left(1 - \frac{\lambda}{\lambda_l}\right) - \dfrac{\beta t^q R^q}{q\, \lambda_l^{q/p}} = \frac{s^p}{p}\, (\lambda_l - \lambda) - \dfrac{\beta s^q}{q},
\end{equation}
where we set~$s = tR/\lambda_l^{1/p}$.

Looking at the maximum attained by the last expression in~\eqref{2.14} over all~$s \ge 0$ and using~\eqref{2.11}, we obtain that
\[
\sup_{u \in X}\, E(u) \le \left(\frac{1}{p} - \frac{1}{q}\right)\!\left(\frac{\lambda_l - \lambda}{\beta^{p/q}}\right)^{q/(q - p)} < c^*.
\]
This gives that condition~\eqref{2.7} is fulfilled in this case.
Hence, we can use
Theorem~\ref{Theorem 2.6} and thus obtain~$m$ distinct pairs of nontrivial critical points of~$E$, as desired.
\end{proof}

\section{Uniform Sobolev embeddings}\label{SOBOLEV}

Here we discuss some uniform embedding of Sobolev type,
whose interest may possibly go even beyond the specific goals of this paper
(and, in relation to this, it is a pleasure to thank
Oscar Dom\'{\i}nguez and Petru Mironescu for sharing information about the state of the art on the fractional Sobolev embeddings).

\begin{proposition}\label{PRBMI}
Let~$\Omega$ be a bounded, open subset of~$\R^N$ and~$p\in [1,N)$.

Then, there exists~$C_0=C_0(N,\Omega,p)>0$ such that, for every~$s\in[0,1]$ and every measurable function~$u:\R^N\to\R$ with~$u=0$ a.e. in~$\R^N\setminus\Omega$, one has that
$$ \|u\|_{L^p(\R^N)}\le C_0\,[u]_{s,p}.$$
\end{proposition}

\begin{proof} Suppose not. Then, there exist sequences~$s_k\in[0,1]$ and~$u_k:\R^N\to\R$ with~$u_k=0$ a.e. in~$\R^N\setminus\Omega$ and such that
\begin{equation}\label{NAR} \|u_k\|_{L^p(\R^N)}^p\ge k\,[u_k]_{s_k,p}^p.\end{equation}
Since~$[u_k]_{0,p}=\|u_k\|_{L^p(\R^N)}$, we have that~$s_k\ne0$.
Also, from the classical Sobolev-Poincar\'e Inequality, we have that~$s_k\ne1$.

Up to replacing~$u_k$ by~$u_k/\|u_k\|_{L^p(\R^N)}$, we can assume that
\begin{equation}\label{Normaliz}
\|u_k\|_{L^p(\R^N)}=1.
\end{equation}
We also take a cube~$Q\subset\R^N$ sufficiently large such that~$\Omega\subset Q$. In this way, if
$$ c_k:=\frac1{|Q|}\int_Q u_k(x)\,dx,$$
we deduce from the H\"older Inequality and~\eqref{Normaliz} that
\begin{equation*}
|c_k|\le\frac1{|Q|}\int_Q |u_k(x)|\,dx\le\frac{\| u_k\|_{L^p(Q)}}{|Q|^{1/p}}\le\frac{1}{|Q|^{1/p}}
\end{equation*}
and therefore, up to a subsequence, $c_k$ converges to some~$c\in\R$ as~$k\to+\infty$.
Up to replacing~$u_k$ with~$-u_k$, we may also suppose that~$c\ge0$.

Additionally, we have that
\begin{equation}\label{AJSMr} [u_k]_{s_k,p}^p\ge c_0\,s_k(1-s_k)\iint_{\R^{2N}}\frac{|u_k(x)-u_k(y)|^p}{|x-y|^{N+s_kp}}\,dx\,dy\end{equation}
for some~$c_0=c_0(N,p)>0$.

Accordingly, if~$R>0$ is large enough such that~$Q\subset B_R$, using that~$u=0$ a.e. outside~$B_R$, recalling~\eqref{Normaliz} we find that, if~$s_k\in(0,1)$,
\begin{equation}\label{NARR2}\begin{split}& [u_k]_{s_k,p}^p\ge c_0\,s_k(1-s_k)\iint_{\Omega\times(\R^N\setminus B_R)}\frac{|u_k(x)|^p}{|x-y|^{N+s_kp}}\,dx\,dy\\&\qquad\qquad\qquad\ge
c_0\,s_k(1-s_k)\,\|u_k\|_{L^p(\Omega)}^p
\int_{\R^N\setminus B_{2R}}\frac{dz}{|z|^{N+s_kp}}=c_1\,(1-s_k),\end{split}\end{equation}
with~$c_1=c_1(N,\Omega,p)>0$.

By~\eqref{NAR}, \eqref{Normaliz} and~\eqref{NARR2}, it follows that
\begin{eqnarray*}
\frac1k\ge [u_k]_{s_k,p}^p\ge c_1\,(1-s_k),
\end{eqnarray*}
therefore~$s_k\to1$ as~$k\to+\infty$.

Now we recall\footnote{We think that there is a very small typo in~\cite[Theorem~1]{MR1945278}.
Namely, the condition~$sp<1$ there should read~$sp$ less than the dimension, thus allowing the use of this result in our context.
See also~\cite[formula~(1)]{MR1940355}.} Theorem~1 in~\cite{MR1945278}, according to which
$$ \iint_{Q\times Q}\frac{|u_k(x)-u_k(y)|^p}{|x-y|^{N+s_kp}}\,dx\,dy\ge c_2
\frac{(N-s_k p)^{p-1}}{1-s_k}\| u_k-c_k\|^p_{L^{Np/(N-s_kp)}(Q)},$$
for some~$c_2=c_2(N,\Omega)>0$.

{F}rom this, \eqref{NAR}, \eqref{Normaliz} and~\eqref{AJSMr}, we arrive at
\begin{eqnarray*}&&
\frac1k\ge [u_k]_{s_k,p}^p\ge
c_0\,s_k(1-s_k)\iint_{\R^{2N}}\frac{|u_k(x)-u_k(y)|^p}{|x-y|^{N+s_kp}}\,dx\,dy\\&&\qquad\quad\ge
c_0\,c_2\,(N-s_k p)^{p-1}\,s_k\| u_k-c_k\|^p_{L^{Np/(N-s_kp)}(Q)}.
\end{eqnarray*}
Hence, using the H\"older Inequality with exponents~$\frac{N}{N-s_kp}$ and~$\frac{N}{s_kp}$,
\begin{eqnarray*}&&
\frac1k\ge \frac{c_0\,c_2\,(N-s_k p)^{p-1}\,s_k}{|Q|^{s_kp/N}}\,\| u_k-c_k\|^p_{L^{p}(Q)}.
\end{eqnarray*}

As a consequence, taking the limit as~$k\to+\infty$, we obtain that
$$ (N-p)^{p-1}\,\lim_{k\to+\infty}\| u_k-c_k\|^p_{L^{p}(Q)}=0,$$
giving that~$u_k-c_k\to0$ in~$L^p(\Omega)$.

Thus, for all~$\rho>0$ and~$x_0\in\R^N$, we define
$$ m_{k,x_0,\rho}:=\frac{1}{|B_\rho|}\int_{B_\rho(x_0)} u_k(x)\,dx=\frac{1}{|B_\rho|}\int_{B_\rho(x_0)\cap \Omega} u_k(x)\,dx$$
and we have that
\begin{eqnarray*}&&
\lim_{k\to+\infty} \left|m_{k,x_0,\rho}-\frac{c\,|B_\rho(x_0)\cap\Omega|}{|B_\rho|}\right|=
\frac{1}{|B_\rho|}\lim_{k\to+\infty} \left|\,\int_{B_\rho(x_0)\cap\Omega}\big( u_k(x)-c\big)\,dx\right|\\&&\qquad\le
\frac{1}{|B_\rho|}\lim_{k\to+\infty} \left[\,\int_{B_\rho(x_0)\cap\Omega}\big| u_k(x)-c_k\big|\,dx+
\int_{B_\rho(x_0)\cap\Omega}\big| c_k-c\big|\,dx
\right]=0.
\end{eqnarray*}

Consequently,
\begin{eqnarray*}&&
\lim_{k\to+\infty} \|u_k-m_{k,x_0,\rho}\|_{L^p(B_\rho(x_0))}^p\\&&\qquad=
\lim_{k\to+\infty}
\left(\,\int_{B_\rho(x_0)\cap\Omega} \big|u_k(x)-m_{k,x_0,\rho}\big|^p\,dx+
\int_{B_\rho(x_0)\setminus\Omega} |m_{k,x_0,\rho}|^p\,dx\right)\\&&\qquad=
\int_{B_\rho(x_0)\cap\Omega} \left|c-\frac{c\,|B_\rho(x_0)\cap\Omega|}{|B_\rho|}\right|^p\,dx+
\int_{B_\rho(x_0)\setminus\Omega}
\frac{c^p\,|B_\rho(x_0)\cap\Omega|^p}{|B_\rho|^p}\,dx\\&&\qquad=c^p\,
\left(\frac{|B_\rho(x_0)\setminus\Omega|^p\, |B_\rho(x_0)\cap\Omega|}{|B_\rho|^p}+
\frac{|B_\rho(x_0)\cap\Omega|^p\, |B_\rho(x_0)\setminus\Omega|}{|B_\rho|^p}\right).
\end{eqnarray*}

Furthermore, by~\cite[Theorem~2.4]{MR4225499},
$$ \sup_{{x_0\in\R^N}\atop{\rho>0}}\rho^{-s_k}\|u_k-m_{k,x_0,\rho}\|_{L^p(B_\rho(x_0))}
\le \widehat{C}\,[u_k]_{s_k,p},$$
for some~$\widehat{C}=\widehat{C}(N,p)>0$.

{F}rom these observations, we obtain that, for every~$\rho>0$ and~$x_0\in\R^N$,
\begin{eqnarray*}&&
\frac{c}\rho\,
\left(\frac{|B_\rho(x_0)\setminus\Omega|^p\, |B_\rho(x_0)\cap\Omega|}{|B_\rho|^p}+
\frac{|B_\rho(x_0)\cap\Omega|^p\, |B_\rho(x_0)\setminus\Omega|}{|B_\rho|^p}\right)^{1/p}\\&&\qquad=
\lim_{k\to+\infty}\frac1\rho\,\|u_k-m_{k,x_0,\rho}\|_{L^p(B_\rho(x_0))}=
\lim_{k\to+\infty}
\rho^{-s_k}\|u_k-m_{k,x_0,\rho}\|_{L^p(B_\rho(x_0))}\\&&\qquad
\le \widehat{C}\,\lim_{k\to+\infty}[u_k]_{s_k,p}\le \widehat{C}\,\lim_{k\to+\infty}\frac1{k^{1/p}}=0.
\end{eqnarray*}
This gives that~$c=0$, and accordingly that~$u_k\to0$ in~$L^p(\Omega)$, but this is in contradiction with~\eqref{Normaliz}.
\end{proof}

\begin{theorem}\label{SOBOLEVEMBEDDING}
Let~$\Omega$ be a bounded, open subset of~$\R^N$ and~$p\in (1,N)$.

Then, there exists~$C=C(N,\Omega,p)>0$ such that, for every~$s$, $S\in[0,1]$ with~$s\le S$ and every measurable function~$u:\R^N\to\R$ with~$u=0$ a.e. in~$\R^N\setminus\Omega$, one has that
$$ [u]_{s,p}\le C\,[u]_{S,p}.$$
\end{theorem}

\begin{proof} Let us first suppose that~$S=1$. In this case, we can also assume that~$\|\nabla u\|_{L^p(\R^N)}<+\infty$, otherwise the desired result is obvious.

Also, if~$S=1$ and~$s=1$, the desired result is obvious, and
if~$S=1$ and~$s=0$, then the desired result follows from Proposition~\ref{PRBMI}.

If~$S=1$ and~$s\in(0,1)$, then we argue as follows. For all~$x$, $y\in\R^N$ with~$|x-y|<1$ we have that
\begin{eqnarray*}
|u(x)-u(y)|\le |x-y|\int_0^1 |\nabla u(tx+(1-t)y)|\,dt
\end{eqnarray*}
and thus, using the substitutions~$z:=x-y$ and~$w:=tx+(1-t)y$, and noticing that~$dz\,dw=dx\,dy$,
\begin{eqnarray*}
&&\iint_{\R^{2N}\cap\{|x-y|<1\}}\frac{|u(x)-u(y)|^p}{|x-y|^{N+sp}}\,dx\,dy
\le\int_{\R^{2N}\cap\{|z|<1\}} |z|^{p(1-s)-N}\left(\int_0^1 |\nabla u(w)|\,dt\right)^p\,dz\,dw
\\&&\qquad\qquad\le\frac{C_0}{1-s}\|\nabla u\|_{L^p(\R^N)}^p,
\end{eqnarray*}
for some~$C_0=C_0(N,p)>0$.

Consequently,
\begin{eqnarray*}&&
\iint_{\R^{2N}}\frac{|u(x)-u(y)|^p}{|x-y|^{N+sp}}\,dx\,dy\\
&&\qquad\le\frac{C_0}{1-s}\|\nabla u\|_{L^p(\R^N)}^p
+\iint_{\R^{2N}\cap\{|x-y|\ge1\}}\frac{(|u(x)|+|u(y)|)^p}{|x-y|^{N+sp}}\,dx\,dy\\
&&\qquad\le\frac{C_0}{1-s}\|\nabla u\|_{L^p(\R^N)}^p+\frac{C_1}{s}\| u\|_{L^p(\R^N)}^p,
\end{eqnarray*}
for some~$C_1=C_1(N,p)>0$.

This and the classical Sobolev-Poincar\'e Inequality yield that
$$ \iint_{\R^{2N}}\frac{|u(x)-u(y)|^p}{|x-y|^{N+sp}}\,dx\,dy\le C_2\left(
\frac{1}{1-s}+\frac1s\right)\|\nabla u\|_{L^p(\R^N)}^p$$
for some~$C_2=C_2(N,\Omega,p)>0$ and the desired result follows.

Now we suppose~$S\ne1$. If~$s=0$, the desired result follows from Proposition~\ref{PRBMI}.
If instead~$s\in(0,1)$, we use~\cite[Theorem~3.9]{MR4525724} with~$\alpha:=1$ and we find that
\begin{eqnarray*}&& \|u\|_{L^p(\R^N)}+\left(\min\{s,1-s\} \iint_{\R^{2N}}\frac{|u(x)-u(y)|^p}{|x-y|^{N+sp}}\,dx\,dy\right)^{1/p}
\\&&\qquad\le
C_3\,\left(\|u\|_{L^p(\R^N)}+\left(\min\{S,1-S\} \iint_{\R^{2N}}\frac{|u(x)-u(y)|^p}{|x-y|^{N+Sp}}\,dx\,dy
\right)^{1/p}\right),\end{eqnarray*}
for some~$C_3=C_3(N,p)>0$, and therefore
\begin{equation*}
[u]_{s,p}\le\|u\|_{L^p(\R^N)}+[u]_{s,p}\le C_4\,\Big(\|u\|_{L^p(\R^N)}+[u]_{S,p}\Big),
\end{equation*}
for some~$C_4=C_4(N,p)>0$.

This and Proposition~\ref{PRBMI} give the desired result.
\end{proof}

\section{Functional setting towards the proof of Theorem~\ref{mainp1}}\label{BaUPS}

Here we construct suitable functional spaces which are helpful to prove Theorem~\ref{mainp1}.

For this, given a measure~$\mu^+$ satisfying~\eqref{mu00}, we set
$$ \rho_p(u):=\left( \,\int_{[0,1]}[u]_{s,p}^p \,d\mu^+(s)\right)^{1/p}$$
and define~$\mathcal{X}_p(\Omega)$ as the set of measurable functions~$u:\R^N\to\R$ such that~$u=0$
in~$\R^N\setminus\Omega$ and~$\rho_p(u)<+\infty$.

In this setting, we can ``reabsorb'' the negative part of the signed measure~$\mu$:

\begin{proposition}\label{absorb} Let $p\in (1, N)$ and assume that~\eqref{mu3} and~\eqref{mu2} hold.

Then, there exists~$c_0=c_0(N,\Omega, p)>0$ such that,
for any~$u\in\mathcal{X}_p(\Omega)$, we have
\[
\int_{[{{ 0 }}, \overline s]} [u]_{s, p}^p \, d\mu^- (s) \le c_0\,\gamma \int_{[\overline s, 1]} [u]^p_{s, p} \, d\mu(s).
\]
\end{proposition}

\begin{proof}
We observe that if~$\mu^+([\overline s,1])=0$, condition~\eqref{mu2} implies that~$\mu^-([0,\overline s])=0$ and then the result is proved.

Thus, for the remaining part of this proof we assume that~$\mu^+([\overline s,1])>0$.

By using Theorem~\ref{SOBOLEVEMBEDDING} with $s$ and $S:=\overline s$ we get,
\[
[u]_{s, p}\le C(N,\Omega,p) [u]_{\overline s, p}.
\]
Furthermore,  employing Theorem~\ref{SOBOLEVEMBEDDING} with $s:=\overline s$ and $S:= s$, we have
\[
[u]_{\overline s, p}\le C(N,\Omega,p) [u]_{s, p}.
\]
Hence, by using the previous inequalities together with the assumptions~\eqref{mu3}
and~\eqref{mu2}, we obtain that
\begin{eqnarray*}
&& \int_{[{{{0}}}, \overline s]} [u]_{s, p}^p \, d\mu^-(s) \le C^p(N,\Omega,p)
\int_{[{{{0}}}, \overline s]} [u]_{\overline s, p}^p \, d\mu^-(s) 
= C^p(N,\Omega,p)\, [u]_{\overline s, p}^p \,\mu^-([{{{0}}}, \overline s]) \\
&&\qquad \le C^p(N,\Omega,p)\gamma
\, [u]_{\overline s, p}^p \,\mu^+([ \overline s,1]) =
C^p(N,\Omega,p)\gamma
\int_{[\overline s,1]} [u]_{\overline s, p}^p \, d\mu^+(s)\\
&&\qquad\le C^{2p}(N,\Omega,p)\gamma
\int_{[\overline s,1]} [u]_{ s, p}^p \, d\mu^+(s) = C^{2p}(N,\Omega,p)\gamma \int_{[\overline s,1]} [u]_{ s, p}^p \, d\mu(s),
\end{eqnarray*}
which gives the desired result with $c_0:= C^{2p}(N,\Omega,p)$.
\end{proof}

We notice that
\begin{equation}\label{complete00}
{\mbox{$\mathcal X_p(\Omega)$ is complete}}\end{equation}
(the case~$p=2$ being proved in~\cite{CATERINA},
the general case being similar).

Also, we have that:

\begin{lemma}\label{UNCONCVE} $\mathcal X_p(\Omega)$ is
a uniformly convex  space.\end{lemma}

\begin{proof} We need to prove that
for every~$\eps\in(0,2]$ there exists~$\delta>0$ 
such that if~$u$, $v\in{\mathcal{X}}_p(\Omega)$ are such that~$\rho_p(u)=\rho_p(v)=1$ and~$\rho_p(u-v)\geq\eps$,
then~$\rho_p(u+v)\le2-\delta$.

To this end, 
we modify some classical approaches
(see e.g.~\cite{MR748950, MR3859645} and the references therein). The details are not completely obvious and go as follows.

Let us first deal with the case~$p\in[2,+\infty)$. In this case, for all~$a$, $b\in\R$, we have that
\begin{equation}\label{PPlas}
|a+b|^p+|a-b|^p\le2^{p-1}\big( |a|^p+|b|^p\big),
\end{equation}
see~\cite[Theorem~2]{MR1501880}.


Hence, we use~\eqref{PPlas} with
\begin{equation*}
a:=u(x)-u(y)\qquad{\mbox{and}}\qquad b:=v(x)-v(y)\end{equation*} and we find that
\begin{eqnarray*}
[\rho_p(u+v)]^p &=&\int_{[0,1]}\left( c_{N,s,p}\iint_{\R^{2N}}\frac{|u(x)+v(x)-u(y)-v(y)|^p}{|x-y|^{N+sp}}\,dx\,dy\right)\,d\mu^+(s)
\\&\le&
2^{p-1}\Bigg[\;
\int_{[0,1]}\left( c_{N,s,p}\iint_{\R^{2N}}\frac{|u(x)-u(y)|^p}{|x-y|^{N+sp}}\,dx\,dy\right)\,d\mu^+(s)\\&&\qquad\qquad
+
\int_{[0,1]}\left( c_{N,s,p}\iint_{\R^{2N}}\frac{|v(x)-v(y)|^p}{|x-y|^{N+sp}}\,dx\,dy\right)\,d\mu^+(s)
\Bigg]\\&&\qquad
-\int_{[0,1]}\left( c_{N,s,p}\iint_{\R^{2N}}\frac{|u(x)-v(x)-u(y)+v(y)|^p}{|x-y|^{N+sp}}\,dx\,dy\right)\,d\mu^+(s)\\
&=& 2^{p-1}\Big[
[\rho_p(u)]^p+[\rho_p(v)]^p
\Big]-[\rho_p(u-v)]^p\\
&\le& 2^{p}-\eps^p\\&=&(2-\delta)^p,
\end{eqnarray*}
with
$$ \delta=\delta(\eps):=2-\big(2^p-\eps^p\big)^{1/p}.$$
We stress that~$\delta$ is strictly increasing in~$\eps\in(0,2]$, therefore~$\delta>\delta(0)=0$, and the proof
of the uniform convexity in this case is thereby complete.

If instead~$p\in(1,2)$, 
we use that, for every~$\Phi$, $\Psi\in L^p(\R^{2N})$,
\begin{equation}\label{ShbsnTA}\begin{split}&
\left\| \frac{|\Phi|+|\Psi|}2\right\|_{L^p(\R^{2N})}^{2-p}
\left( \frac{\|\Phi\|_{L^p(\R^{2N})}^p+\|\Psi\|_{L^p(\R^{2N})}^p}2-
\left\| \frac{\Phi+\Psi}2\right\|_{L^p(\R^{2N})}^p\right)\\&\qquad\qquad\ge\frac{p(p-1)}8 \|\Phi-\Psi\|_{L^p(\R^{2N})}^2,\end{split}
\end{equation}
see~\cite[Theorem~1]{MR748950}.

We choose
$$ \Phi(x,y,s):=c_{N, s, p}^{\frac1p}\,\frac{u(x)-u(y)}{|x-y|^{\frac{N+sp}p}}\qquad{\mbox{and}}\qquad
\Psi(x,y,s):=c_{N, s, p}^{\frac1p}\,\frac{v(x)-v(y)}{|x-y|^{\frac{N+sp}p}}.$$
In this way,
\begin{eqnarray*}
&&\|\Phi\|_{L^p(\R^{2N}\times [0, 1])}=\rho_p(u)=1,\\
&&\|\Psi\|_{L^p(\R^{2N}\times [0, 1])}=\rho_p(v)=1,\\
&& \|\Phi+\Psi\|_{L^p(\R^{2N}\times [0, 1])}=\rho_p(u+v),\\
&& \|\Phi-\Psi\|_{L^p(\R^{2N}\times [0, 1])}=\rho_p(u-v)\ge\eps\\{\mbox{and }}&&
\| |\Phi|+|\Psi|\|_{L^p(\R^{2N}\times [0, 1])}\le\| \Phi\|_{L^p(\R^{2N}\times [0, 1])}+\|\Psi\|_{L^p(\R^{2N}\times [0, 1])}=\rho_p(u)+\rho_p(v)=2.
\end{eqnarray*}
Therefore, in light of~\eqref{ShbsnTA},
$$ 1-\frac{[\rho_p(u+v)]^p}{2^p}\ge\frac{p(p-1)}8 \eps^2
$$
and, as a result,
\begin{eqnarray*}
\rho_p(u+v) \le 2\left(1-\frac{\eps^2p(p-1)}8 \right)^{1/p}=2-\delta,
\end{eqnarray*}
with
$$ \delta=\delta(\eps):=2\left[1-\left(1-\frac{\eps^2p(p-1)}8 \right)^{1/p}\right].$$
We stress that~$\eps^2p(p-1)<8$ in this case, whence~$\delta$ is strictly increasing in~$\eps\in(0,2]$. This entails that~$\delta>\delta(0)=0$, thus completing the proof
of the uniform convexity property.
\end{proof}

The setting of solutions that we consider here is the one induced by this functional framework, namely:

\begin{definition}\label{wsol}
A weak solution of problem~\eqref{mainp} is a function~$u \in \mathcal{X}_p(\Omega)$ such that\footnote{We stress that
expressions such as
$$ \int_{[{{{0}}}, 1]}\left(\;\iint_{\R^{2N}} \frac{c_{N,s,p}\,|u(x) - u(y)|^{p-2}\, (u(x) - u(y))\, (v(x) - v(y))}{|x - y|^{N+sp}}\, dx dy\right) d\mu(s)$$
are a slight abuse of notation. To be precise, one should write instead
\begin{eqnarray*}&&
\int_{({{{0}}}, 1)}\left(\;\iint_{\R^{2N}} \frac{c_{N,s,p}\,|u(x) - u(y)|^{p-2}\, (u(x) - u(y))\, (v(x) - v(y))}{|x - y|^{N+sp}}\, dx dy\right) d\mu
(s)\\
&&\qquad\qquad+
\mu(0)\int_{\R^N} u(x)\,v(x)\,dx+\mu(1)\int_{\R^N} \nabla u(x)\cdot\nabla v(x)\,dx.\end{eqnarray*}
For the sake of shortness, however, we will accept the above abuse of notation whenever typographically convenient.}
\[
\begin{split}
&\int_{[{{{0}}}, 1]}\left(\;\iint_{\R^{2N}} \frac{c_{N,s,p}\,|u(x) - u(y)|^{p-2}\, (u(x) - u(y))\, (v(x) - v(y))}{|x - y|^{N+sp}}\, dx dy\right) d\mu^+(s)\\[5pt]
&-\int_{[{{{0}}}, \overline s]}\left(\;\iint_{\R^{2N}} \frac{c_{N,s,p}\,|u(x) - u(y)|^{p-2}\, (u(x) - u(y))\, (v(x) - v(y))}{|x - y|^{N+sp}}\, dx dy\right) d\mu^-(s)\\[5pt]
&= \lambda \int_\Omega |u|^{p-2}\, uv\, dx + \int_\Omega |u|^{p_{s_\sharp}^* - 2}\, uv\, dx,
\end{split}
\]
for all~$v \in \mathcal{X}_p(\Omega)$.
\end{definition}

Solutions of problem~\eqref{mainp} coincide with the critical points of the functional~$E: \mathcal X_p(\Omega)\to\R$ given by
\begin{equation}\label{funp}
E_p(u) =\frac1p [\rho_p(u)]^p -\frac1p\int_{[{{{0}}}, \overline s]} [u]_{s, p}^p \, d\mu^-(s) - \frac{\lambda}{p} \int_\Omega |u|^p \, dx - \frac{1}{p_{s_\sharp}^*} \int_\Omega |u|^{p_{s_\sharp}^*}\, dx.
\end{equation}

\section{Proof of Theorem~\ref{mainp1}}\label{BaUPS-2}

We now prove Theorem~\ref{mainp1}. The strategy is to reduce ourselves to the abstract setting introduced in Section~\ref{AB2}
and exploit Theorem~\ref{Theorem 2.9}.

To this end, as already anticipated in Section~\ref{AB2}, we choose 
\begin{equation}\label{SETTING}\begin{split}
&W:={\mathcal{X}}_p(\Omega),\qquad A_pu:=\int_{[0, 1]} (-\Delta)_p^{s} u \, d\mu^+(s),\qquad B_pu:=|u|^{p-2}u,\\&
L_pu:=\int_{[0, 1]} (-\Delta)_p^{s} u \, d\mu^-(s)
\qquad{\mbox{and}}\qquad f(u):= |u|^{p_{s_\sharp}^* - 2}\, u .
\end{split}\end{equation}
Our goal is now to systematically check that the abstract hypotheses requested in Section~\ref{AB2} are fulfilled.
Specifically, we need to check~\eqref{UNICO}, \eqref{UNICO2}, \eqref{2.888-0}, \eqref{2.888-01}, \eqref{lppotop}, as well as the structural conditions~$(A_1)$, $(A_2)$, $(B_1)$, $(B_2)$, $(B_3)$, $(L_1)$, $(\mathcal F_1)$, $(\mathcal F_2)$, $(\mathcal F_3)$ and~$(\mathcal N_1)$.\medskip

Let us proceed in order. 
First of all, by~\eqref{complete00} and Lemma~\ref{UNCONCVE}, we have that~$\mathcal X_p(\Omega)$ is
a uniformly convex Banach space, which is the claim in~\eqref{UNICO}.

Regarding the claim in~\eqref{UNICO2}, we first need to interpret~$A_p$ and~$B_p$
as operators from~$W$ to its dual, namely we rephrase~\eqref{SETTING}, given~$u$, $v\in{\mathcal{X}}_p(\Omega)$, as
\begin{eqnarray*}&&
( A_pu,v)=\int_{\R^N}\left(\;\int_{[0, 1]} (-\Delta)_p^{s} u(x)\,v(x) \, d\mu^+(s)\right)\,dx\\{\mbox{and }}&&(B_pu,v)=\int_{\R^N} |u(x)|^{p-2}u(x)\,v(x)\,dx.
\end{eqnarray*}
We observe that:

\begin{lemma}\label{APC}
$A_p$ is continuous from~$W$ to~$W^*$.\end{lemma}

\begin{proof} We notice that
\begin{equation}\label{hdsu3y259676tgrhueh}\begin{split}
&\int_{\R^N} (-\Delta)_p^{s} u(x)\,v(x)\,dx\\
=\;&
2c_{N,s,p} \,\lim_{\eps \searrow 0} \int_{\R^N}\left[\;
\int_{\R^N \setminus B_\eps(x)} \frac{|u(x) - u(y)|^{p-2}\, (u(x) - u(y)) v(x)}{|x - y|^{N+sp}}\, dy\right]\,dx
\\=\;& {c_{N,s,p}} \,\iint_{\R^{2N}} \frac{|u(x) - u(y)|^{p-2}\, (u(x) - u(y))( v(x)-v(y))}{|x - y|^{N+sp}}\, dx\,dy.
\end{split}\end{equation}
The H\"older Inequality with exponents~$\frac{p}{p-1}$ and~$p$ 
applied to the functions
\begin{equation}\label{2Aalsdx029oiefkvHO-bn}
\frac{c_{N,s,p}^{(p-1)/p}\,|u(x) - u(y)|^{p-2}\, (u(x) - u(y))}{|x - y|^{(N+sp)(p-1)/p}}
\qquad{\mbox{ and }}\qquad
\frac{c_{N,s,p}^{1/p}\,( v(x)-v(y))}{|x - y|^{(N+sp)/p}}\end{equation}
returns that
\begin{equation}\label{2Aalsdx029oiefkvHO}\begin{split}
|(A_pu,v)|&=
\left|\;\;\iint_{[0,1]\times\R^N} (-\Delta)_p^{s} u(x)\,v(x)\,d\mu^+(s)\,dx\right|\\
&\le
\left[\;\; \iiint_{[0,1]\times\R^{2N}} c_{N,s,p}\frac{|u(x) - u(y)|^p}{|x - y|^{N+sp}}\,d\mu^+(s)\, dx\,dy\right]^{(p-1)/p}\\&\qquad\qquad\times
\left[\;\;\iiint_{[0,1]\times\R^{2N}} c_{N,s,p}\frac{| v(x)-v(y)|^p}{|x - y|^{N+sp}}\,d\mu^+(s)\, dx\,dy\right]^{1/p}
\\&=\left(\;\int_{[0,1]} [u]_{s,p}^p\,d\mu^+(s)\right)^{(p-1)/p}\left(\;\int_{[0,1]} [v]_{s,p}^p\,d\mu^+(s)\right)^{1/p}
\\&=[\rho_p(u)]^{p-1}\,\rho_p(v).
\end{split}\end{equation}
This gives that
\begin{equation}\label{Aalsdx029oiefkv}
|(A_pu,v)|\le[\rho_p(u)]^{p-1}\,\rho_p(v),\end{equation}
which yields the desired result.\end{proof}

Moreover, we have that:

\begin{lemma}\label{BPC}
$B_p$ is continuous from~$W$ to~$W^*$.\end{lemma}

\begin{proof} Using the H\"older Inequality with exponents~$\frac{p}{p-1}$ and~$p$, we have that
\begin{eqnarray*}
|( B_pu,v)|\le\int_{\Omega} |u(x)|^{p-1}\,|v(x)|\,dx\le \|u\|_{L^p(\R^N)}^{p-1}\,\|v\|_{L^p(\R^N)}.
\end{eqnarray*}
Consequently, by Proposition~\ref{PRBMI}, for all~$s\in[0,1]$,
$$ |( B_pu,v)|\le\widetilde C\,[u]_{s,p}^{p-1}\,[v]_{s,p},$$
for some~$\widetilde C=\widetilde C(N,\Omega,p)>0$.

Integrating with respect to~$\mu^+$
and using again the H\"older Inequality with exponents~$\frac{p}{p-1}$ and~$p$ we deduce that
$$ |(B_pu,v)|\le\widehat C\,[\rho_p(u)]^{p-1}\,\rho_p(v),$$
for some~$\widehat C=\widehat C(N,\Omega,p,\mu)>0$, which establishes the desired result.
\end{proof}

By Lemmata~\ref{APC} and~\ref{BPC}, combined with the general result in~\eqref{ARETRH}, it follows that~$A_p$ and~$B_p$ are potential operators, and we have thus checked condition~\eqref{UNICO2}.\medskip

As for~\eqref{2.888-0}, again we have to interpret the definition of~$f$ in~\eqref{SETTING} as an operator from~$W$ to its dual,
namely, for all~$u$, $v\in\mathcal{X}_p(\Omega)$,
$$ (f(u),v)=\int_{\R^N} |u(x)|^{p_{s_\sharp}^* - 2}\, u(x)\,v(x)\,dx.$$
In this setting, we have that

\begin{lemma}\label{fcot} $f$ is continuous from~$W$ to~$W^*$.
\end{lemma}

\begin{proof} By the H\"older Inequality with exponents~${p_{s_\sharp}^* }/({p_{s_\sharp}^* -1})$ and~${p_{s_\sharp}^*}$,
we have that
\begin{eqnarray*}
\big|( f(u),v)\big|\le\int_{\R^N} |u(x)|^{p_{s_\sharp}^* -1}\, |v(x)|\,dx\le
\|u\|^{p_{s_\sharp}^*-1}_{L^{p_{s_\sharp}^*}(\R^N)}\,\|v\|_{L^{p_{s_\sharp}^*}(\R^N)}.\end{eqnarray*}
This and the fractional Sobolev embedding (see e.g.~\cite{MR2944369}) gives that
$$ \big|( f(u),v)\big|\le C_\star\,[u]^{p-1}_{s_\sharp,p} \,[v]_{s_\sharp,p},$$
for some~$C_\star=C_\star(N,p,s_\sharp)>0$.

Hence, by Theorem~\ref{SOBOLEVEMBEDDING}, for all~$s\in[s_\sharp,1]$,
$$ \big|( f(u),v)\big|\le C_\sharp\,[u]^{p-1}_{ s,p} \,[v]_{s,p},$$
for some~$C_\sharp=C_\sharp(N,\Omega, p,s_\sharp)>0$.

This and the H\"older inequality with exponents~$\frac{p}{p-1}$ and~$p$ give that
\begin{eqnarray*}&& \mu^+\big([s_\sharp,1]\big)
\big|( f(u),v)\big|\le C_\sharp\,\int_{[s_\sharp,1]}[u]^{p-1}_{ s,p} \,[v]_{s,p}\,d\mu^+(s)\\&&\qquad\le C_\sharp\,\left(\;
\int_{[s_\sharp,1]}[u]^{p}_{ s,p} \,d\mu^+(s)\right)^{(p-1)/p}\left(\;\int_{[s_\sharp,1]} [v]_{s,p}^p\,d\mu^+(s)
\right)^{1/p}\le C_\sharp\,[\rho_p(u)]^{p-1}\,\rho_p(v),\end{eqnarray*}
which gives the desired result.
\end{proof}

\begin{corollary}\label{2.888-0C}
$f$ is a potential operator.
\end{corollary}

\begin{proof} We know by Lemma~\ref{fcot} that~$f\in C(W,W^*)$. A potential of~$f$ is given by \begin{equation}\label{POTEF}F(u):=
\frac1{p_{s_\sharp}^*}
\int_{\R^N} |u(x)|^{p_{s_\sharp}^* }\,dx.\qedhere\end{equation}
\end{proof}

In light of Corollary~\eqref{2.888-0C}, we have checked condition~\eqref{2.888-0}, as desired.\medskip

The fact that~$f$ is odd, as requested in~\eqref{2.888-01} is obvious in our case.
\medskip

As for the assumption on~$L_p$ in~\eqref{lppotop},  we have to interpret the definition of~$L_p$ 
in~\eqref{SETTING} as an operator from~$W$ to its dual,
namely, for all~$u$, $v\in\mathcal{X}_p(\Omega)$,
$$ (L_p(u),v)=\int_{\R^N}\left(\;\int_{[0, 1]} (-\Delta)_p^{s} u(x)\,v(x) \, d\mu^-(s)\right)\,dx.$$
In this case, we have that

\begin{lemma}\label{Lpcont} $L_p$ is continuous from~$W$ to~$W^*$.
\end{lemma}

\begin{proof} We use~\eqref{hdsu3y259676tgrhueh} and argue as in~\eqref{2Aalsdx029oiefkvHO}
with~$A_p$ replaced by~$L_p$ and~$\mu^+$ replaced by~$\mu^-$ to conclude that
$$|(L_pu,v)|\le[\rho_p(u)]^{p-1}\,\rho_p(v),$$
which yields the desired result.
\end{proof}

\begin{corollary}\label{Lppotope000}
$L_p$ is a potential operator.
\end{corollary}

\begin{proof}
Lemma~\ref{Lpcont} informs us that~$L_p\in C(W,W^*)$. Moreover, a potential of~$L_p$ is given by
\begin{equation}\label{Npippoepluto} N_p(u):=\frac1p(L_pu,u)=\frac1p \int_{\R^N}\left(\;\int_{[0, 1]} (-\Delta)_p^{s} u(x)\,u(x) \, d\mu^-(s)\right)\,dx.\qedhere
\end{equation}
\end{proof}

Lemmata~\ref{Lpcont} and~\ref{Lppotope000} establish that the assumption in~\eqref{lppotop} is satisfied.
\medskip

The homogeneity required in~$(A_1)$ is obvious.
In relation to~$(A_2)$, for all~$u, v \in {\mathcal{X}}_p(\Omega)$, the bound on~$\dualp{\Ap[u]}{v}$
is a consequence of~\eqref{Aalsdx029oiefkv}.
In addition, equality holds true in~\eqref{Aalsdx029oiefkv}
if and only if it holds in~\eqref{2Aalsdx029oiefkvHO},
and so if and only if the~$\frac{p}{p-1}$ power of the absolute value of the first function in~\eqref{2Aalsdx029oiefkvHO-bn}
is proportional to the~$p$ power of the absolute value of the second
function in~\eqref{2Aalsdx029oiefkvHO-bn} (and with a consistent sign if
we want~$\dualp{\Ap[u]}{v}\ge0$), i.e. if and only if there exist~$\alpha$, $\beta\ge0$, with~$\alpha^2+\beta^2\ne0$, such that
\begin{equation*}
\alpha\, (u(x) - u(y))= \beta\, (v(x)-v(y)).\end{equation*}
Choosing~$y\in\R^n\setminus\Omega$ (in which case~$u(y)=v(y)=0$), we see that this
condition is equivalent to~$\alpha u = \beta v$.

These observations give that~$(A_2)$ holds true.\medskip

The homogeneity and the odd property in~$(B_1)$ is obvious, and~$(B_2)$ follows from
\begin{equation}\label{NASDL}
(B_pu,u)=\int_{\R^N} |u(x)|^{p}\,dx=\|u\|_{L^p(\R^N)}^p>0,
\end{equation}
unless~$u$ vanishes identically, and
\begin{eqnarray*}
(B_pu,v)\le\int_{\R^N} |u(x)|^{p-1}\,|v(x)|\,dx\le \|u\|_{L^p(\R^N)}^{p-1}\,\|v\|_{L^p(\R^N)}=(B_pu,u)^{(p-1)/p}\,(B_pv,v)^{1/p}.
\end{eqnarray*}

\medskip

With respect to the compactness property requested in~$(B_3)$, it is a consequence of the following result:

\begin{lemma}
Let~$u_n$ be a bounded sequence in~${\mathcal{X}}_p(\Omega)$.
Then, the sequence~$U_n:=B_pu_n$ is precompact
in the dual of~${\mathcal{X}}_p(\Omega)$.
\end{lemma}

\begin{proof} 
We exploit Theorem~\ref{SOBOLEVEMBEDDING} to see that
\begin{eqnarray*}
\int_{[\overline s,1]}[u_n]_{s,p}^p \,d\mu^+(s)
\ge \frac1{C(N,\Omega,p)}\int_{[\overline s,1]}[u_n]_{\overline s,p}^p \,d\mu^+(s)=\frac{\mu^+([\overline s,1])}{C(N,\Omega,p)}[u_n]_{\overline s,p}^p .
\end{eqnarray*}
In light of the assumption on~$\mu$ in~\eqref{mu00}, this implies that~$[u_n]_{\overline s,p}$ is uniformly bounded in~$n$, and therefore 
by the fractional compact embedding (see e.g. 
\cite[Corollary~7.2]{MR2944369}) we obtain that, up to a subsequence,
\begin{equation}\label{fhwerutgvbhsd879999999999962-945}
{\mbox{$u_n$ converges to some~$u\in L^p(\Omega)$.}}\end{equation}

Thus, we set~$U:=B_pu$ and we 
have that, for all~$v\in{\mathcal{X}}_p(\Omega)$,
$$ ( U_n,v)-(U,v)=\int_{\R^N} \Big(|u_n(x)|^{p-2}u_n(x)-|u(x)|^{p-2}u(x)\Big)\,v(x)\,dx.$$
By the H\"older Inequality with exponents~$\frac{p}{p-1}$ and~$p$
and Theorem~\ref{SOBOLEVEMBEDDING}, we obtain that
\begin{equation}\label{sweo0p7i8u6t453rkjtrg}\begin{split}
|( U_n,v)-(U,v)|\le\;&
\left(\;\int_{\R^N} \Big|\,|u_n(x)|^{p-2}u_n(x)-|u(x)|^{p-2}u(x)\Big|^{\frac{p}{p-1}}\,dx\right)^{\frac{p-1}{p}} \|v\|_{L^p(\R^N)}\\
\le\;& C \left(\;\int_{\R^N} \Big|\,|u_n(x)|^{p-2}u_n(x)-|u(x)|^{p-2}u(x)\Big|^{\frac{p}{p-1}}\,dx\right)^{\frac{p-1}{p}} \rho_p(v),
\end{split}\end{equation}
where~$C=C(N, \Omega, p,\mu^+)>0$.

We point out that if~$p=2$ the desired result follows from this and~\eqref{fhwerutgvbhsd879999999999962-945}.
If instead~$p\neq2$ some more algebraic manipulations are needed to establish the desired convergence result. Hence, from now on, we suppose that~$p\neq2$.

When~$p>2$, we utilize the following inequality
\begin{equation}\label{dkoejtiohnytkvfxdls}
\big|\,|a|^{p-2}a-|b|^{p-2}b\,\big|\le  (p-1) \big( |a|^{p-2} +|b|^{p-2}\big) |a-b|,
\end{equation}
which holds true for all~$a$, $b\in\R$. 
Indeed, we suppose that~$a<b$ and
we have that
\begin{eqnarray*}&&
\big|\,|a|^{p-2}a-|b|^{p-2}b\,\big|=\left|\int_a^b (p-1) |t|^{p-2}\,dt
\right|\le (p-1) \big( |a|^{p-2} +|b|^{p-2}\big) |a-b|,
\end{eqnarray*}
which is~\eqref{dkoejtiohnytkvfxdls}.

Hence, using formula~\eqref{dkoejtiohnytkvfxdls}
with~$a:=u_n(x)$ and~$b:=u(x)$, we conclude that
\begin{equation}\label{djweioguoe9786543e}\begin{split}
&\int_{\R^N} \Big|\,|u_n(x)|^{p-2}u_n(x)-|u(x)|^{p-2}u(x)\Big|^{\frac{p}{p-1}}\,dx
\\ \le\;& C\int_{\R^N} 
\big( |u_n(x)|^{p-2} +|u(x)|^{p-2}\big)^{\frac{p}{p-1}}
| u_n(x) -u(x)|^{\frac{p}{p-1}}\,dx
\end{split}\end{equation}
for some~$C=C(p)>0$.

In this case, we have that~$p-1>1$ and therefore we can use
the H\"older Inequality with exponents~$\frac{p-1}{p-2}$ and~$p-1$
and obtain from~\eqref{djweioguoe9786543e} that
\begin{equation*}\begin{split}
&\int_{\R^N} \Big|\,|u_n(x)|^{p-2}u_n(x)-|u(x)|^{p-2}u(x)\Big|^{\frac{p}{p-1}}\,dx
\\ \le\;& C\left(\;\int_{\R^N} 
\big( |u_n(x)|^{p-2} +|u(x)|^{p-2}\big)^{\frac{p}{p-2}}\,dx\right)^{\frac{p-2}{p-1}}
\left(\;\int_{\R^N}
| u_n(x) -u(x)|^{p}\,dx\right)^{\frac{1}{p-1}}
\\ \le\;& C\left(\;\int_{\R^N} 
\big( |u_n(x)|^{p} +|u(x)|^{p}\big)\,dx\right)^{\frac{p-2}{p-1}}
\left(\;\int_{\R^N}
| u_n(x) -u(x)|^{p}\,dx\right)^{\frac{1}{p-1}},
\end{split}\end{equation*}
up to renaming~$C$.

The desired result in the case~$p>2$ then follows from this, \eqref{fhwerutgvbhsd879999999999962-945}
and~\eqref{sweo0p7i8u6t453rkjtrg}.

If instead~$p\in(1,2)$, we first observe that, for all~$t\in\R\setminus\{0\}$,
\begin{equation}\label{pso2}
\frac{\big| |1+t|^{p-2}(1+t)-1\big|}{|t|^{p-1}}\le C_\star,
\end{equation}
for some~$C_\star>0$, depending only on~$p$.

To check this, one can define~$\phi(t)$ as the left hand side of~\eqref{pso2} and notice that
$$ \lim_{t\to\pm\infty}\phi(t)=1.$$
Moreover, by L'{H}\^{o}pital's Rule, since~$p<2$,
$$ \lim_{t\to0}\phi(t)=\lim_{t\to0} \frac{\big| (1+t)^{p-1}-1\big|}{|t|^{p-1}}=\lim_{t\to0} 
\big( {\rm sign}(t)\big)^p\;
\frac{ (1+t)^{p-1}-1}{t^{p-1}}
=\pm\lim_{t\to0} \frac{ (1+t)^{p-2}}{t^{p-2}}=0.
$$
These observations establish~\eqref{pso2}.

Moreover, in this case, for all~$a$, $b\in\R$,
\begin{equation}\label{pso23}
\big|\,|a|^{p-2}a-|b|^{p-2}b\,\big|\le  C_\star \,|a-b|^{p-1}.
\end{equation}
Indeed, if~$a=b$, or~$a=0$, or~$b=0$ this is obvious, hence we can suppose that~$a\ne b$, that~$a\ne0$ and that~$b\ne0$.

Also, we can assume that~$a$ and~$b$ have the same sign, because if, say, $a>0>b$, then
$$ \big|\,|a|^{p-2}a-|b|^{p-2}b\,\big|=
a^{p-1}+|b|^{p-1}\le (a+|b|)^{p-1}+(a+|b|)^{p-1}
=2(a-b)^{p-1}=2|a-b|^{p-1}$$
and we are done.

Hence, up to swapping the signs of~$a$ and~$b$, which would not change the desired claim, we can assume, without loss of generality, that~$a$, $b>0$.

Accordingly, we can define
$$ t:=\frac{a}{b}-1.$$
Notice that~$t\ne0$ and then, in view of~\eqref{pso2},
\begin{eqnarray*}
C_\star\ge\frac{\big| |1+t|^{p-2}(1+t)-1\big|}{|t|^{p-1}}
= \frac{\big| (a/b)^{p-1}-1\big|}{|(a/b)-1|^{p-1}}= \frac{\big| a^{p-1}-b^{p-1}\big|}{|a-b|^{p-1}},
\end{eqnarray*}
which establishes~\eqref{pso23}.

As a consequence, owing to~\eqref{sweo0p7i8u6t453rkjtrg} and~\eqref{pso23},
\begin{eqnarray*}
|( U_n,v)-(U,v)|&\le&
C\,C_\star\,\left(\;\int_{\R^N} |u_n(x)-u(x)|^p\,dx\right)^{\frac{p-1}{p}} \rho_p(v).
\end{eqnarray*}
This and the convergence of~$u_n$ in~$L^p(\Omega)$ prove the desired result.
\end{proof}

\medskip

The homogeneity assumption in~$(L_1)$ is also obvious.

\medskip

With respect to the validity of~$(\mathcal F_1)$, we note that the potential~$F$ is given by~\eqref{POTEF} and satisfies
\begin{equation}\label{FDCP}
|F(u)|=\frac1{p_{s_\sharp}^*}\|u\|^{p_{s_\sharp}^* }_{L^{p_{s_\sharp}^* }(\R^N)}.
\end{equation}

Besides, from the fractional Sobolev embedding (see e.g. 
\cite[Theorem~6.5]{MR2944369}) and Theorem~\ref{SOBOLEVEMBEDDING}, for all~$s\in[s_\sharp,1]$,
$$ \|u\|_{L^{p_{s_\sharp}^* }(\R^N)}\le C_1\,[u]_{s_\sharp,p}\le C_2\,[u]_{s,p},$$
for suitable~$C_1=C_1(N,s_\sharp,\Omega,p)>0$ and~$C_2=C_2(N,s_\sharp,\Omega,p)>0$.

Integrating over~$s$ we find that
$$ \mu^+\big([s_\sharp,1]\big)\|u\|_{L^{p_{s_\sharp}^* }(\R^N)}^p\le C_2^p\int_{[s_\sharp,1]}[u]_{s,p}^p\,d\mu^+(s)\le
\big[ C_2\,\rho_p(u)\big]^p.$$

This and~\eqref{FDCP} give that
$$ |F(u)|\le C_3\,[\rho(u)]^{p_{s_\sharp}^*/p},$$
with~$C_3=C_3(N,s_\sharp,\Omega,p,\mu)>0$ and , since~$p_{s_\sharp}^*>p$,
we see that~$(\mathcal F_1)$ is satisfied in our setting.\medskip

In regard to~$(\mathcal F_2)$, one uses~\eqref{SETT1}, \eqref{POTEF} and~\eqref{NASDL} and sees that
\begin{eqnarray*}
F(u)-\frac\beta{q}\big(pJ_p(u)\big)^{q/p}&=&\frac1{p_{s_\sharp}^*}\;
\int_{\R^N} |u(x)|^{p_{s_\sharp}^* }\,dx-\frac{\beta}q\big( \dualp{\Bp[u]}{u}\big)^{q/p}\\
&=&\frac1{p_{s_\sharp}^*}\;\int_{\Omega} |u(x)|^{p_{s_\sharp}^* }\,dx- \frac{\beta}q\left(\int_{\Omega} |u(x)|^{p}\,dx\right)^{q/p}.\end{eqnarray*}
Hence, choosing
\begin{equation}\label{plutoepaper}
q:=p_{s_\sharp}^* >p \qquad {\mbox{and}}\qquad \beta:=\frac1{|\Omega|^{(p_{s_\sharp}^* -p)/p }},\end{equation} the H\"older Inequality with exponents~$p_{s_\sharp}^* /p$
and~$p_{s_\sharp}^* /(p_{s_\sharp}^* -p)$ gives that
$$ F(u)-\frac\beta{q}\big(pJ_p(u)\big)^{q/p}\ge\frac1{p_{s_\sharp}^*}\;
\int_{\Omega} |u(x)|^{p_{s_\sharp}^* }\,dx- \frac{\beta\,|\Omega|^{(p_{s_\sharp}^* -p)/p 
}}{p_{s_\sharp}^* } \int_{\Omega} |u(x)|^{p_{s_\sharp}^*}\,dx=0
,$$
establishing~$(\mathcal F_2)$.
\medskip

Referring to~$(\mathcal F_3)$, we first point out the following weak convergence result:

\begin{lemma}\label{Vitali}
Let~$u_n$ be a bounded sequence in~$\mathcal X_p(\Omega)$. 

Then, there exists~$u\in\mathcal X_p(\Omega)$ such that,
for any~$v\in\mathcal X_p(\Omega)$, we have
\begin{equation}\label{Vconv}
\begin{split}
&\lim_{n\to+\infty}\int_{[{{ 0 }}, 1]} \left(\;\iint_{\R^{2N}} \frac{|u_n(x)-u_n(y)|^{p-2}(u_n(x)-u_n(y)) (v(x)-v(y))}{|x-y|^{N+sp}} \, dxdy\right)\, d\mu^\pm(s)\\
&\qquad=\int_{[{{ 0 }}, 1]} \left(\;\iint_{\R^{2N}} \frac{|u(x) - u(y)|^{p-2}(u(x)-u(y)) (v(x)-v(y))}{|x-y|^{N+sp}} \, dxdy\right)\, d\mu^\pm(s) .
\end{split}
\end{equation}
\end{lemma}

\begin{proof} On the one hand, by~\eqref{complete00} and Lemma~\ref{UNCONCVE}, we have that~$\mathcal X_p(\Omega)$ is complete and uniformly convex.
On the other hand,
the Milman-Pettis Theorem (see e.g.~\cite{MR100215})
states that every uniformly convex Banach space is reflexive. Therefore, $\mathcal X_p(\Omega)$ is reflexive (and the same holds if one replaces~$\mu^+$ by~$\mu^-$).

Hence, by Kakutani's Theorem we know that
the closed balls of~${\mathcal{X}}_p(\Omega)$ are compact in the weak topology, which yields the desired result.
\end{proof}

Moreover, to establish the threshold $c^*$ for the condition~$(\mathcal F_3)$ to hold, we will also need the following Br\'{e}zis--Lieb type result (see~\cite{MR699419}
for the original statement):

\begin{lemma}\label{Brezis-Lieb}
Let~$u_n$ be a bounded sequence in~$\mathcal X_p(\Omega)$.
Suppose that~$u_n$ converges to some~$u$ a.e. in~$\R^N$ as~$n\to +\infty$.

Then,
\begin{equation}\label{djieow34LLLtyb5o4yo3tu3493pp}
\int_{[{{ 0 }}, 1]} [u]^p_{s, p}\, d\mu^{\pm}(s) = \lim_{n\to +\infty}\left(\;\int_{[{{ 0 }}, 1]} [u_n]^p_{s, p}\, d\mu^{\pm}(s) - \int_{[{{ 0 }}, 1]} [u_n -u]^p_{s, p}\, d\mu^{\pm}(s) \right).
\end{equation}
\end{lemma}

\begin{proof}
The aim is to apply~\cite[Theorem 1.9]{MR1817225} to the measure space given by
$(\R^{2N}\times [0, 1], \,dx\,dy\, d\mu^+(s))$ and the functions
\[
f_n(x, y, s):= c_{N, s, p}^{1/p} \frac{u_n(x) -u_n(y)}{|x-y|^{\frac{N+sp}{p}}}\qquad\text{ and }\qquad f(x, y, s):= c_{N, s, p}^{1/p} \frac{u(x) -u(y)}{|x-y|^{\frac{N+sp}{p}}}.
\]
We observe that, since~$u_n$ is a bounded sequence in~$\mathcal X_p(\Omega)$,
there exists~$C>0$, independent of~$n$, such that
\begin{eqnarray*}&& C\ge \int_{[0,1]}[u_n]_{s,p}^p \,d\mu^+(s)
=\int_{[0,1]} c_{N, s, p}\iint_{\R^{2N}} \frac{|u_n(x) -u_n(y)|^p}{|x-y|^{N+sp} }\,dx\,dy\,d\mu^+(s)
\\&&\qquad\qquad
= \iiint_{[0,1]\times\R^{2N}} |f_n(x, y, s)|^p\,dx\,dy\,d\mu^+(s)
,\end{eqnarray*}
which implies that~$f_n$ is uniformly bounded in $L^p\big(
\R^{2N}\times [0, 1], \,dx\,dy\, d\mu^+(s)\big)$.

Moreover, since~$u_n$ converges to~$u$ a.e. in~$\R^N$ as~$n\to +\infty$, we have that~$f_n$ converges to~$f$ a.e. in~$\R^{2N}\times[0,1]$ as~$n\to +\infty$. 

Hence, the assumptions in~\cite[Theorem 1.9]{MR1817225} 
are fulfilled, and therefore
\begin{eqnarray*}&&
\int_{[{{ 0 }}, 1]} [u]^p_{s, p}\, d\mu^{+}(s)
= \iiint_{[0,1]\times\R^{2N}} |f(x, y, s)|^p\,dx\,dy\,d\mu^+(s)\\&
=& \lim_{n\to +\infty}\left(\;\;\iiint_{[0,1]\times\R^{2N}} |f_n(x, y, s)|^p\,dx\,dy\,d\mu^+(s)
- \iiint_{[0,1]\times\R^{2N}} |f(x, y, s)-f_n(x,y,s)|^p\,dx\,dy\,d\mu^+(s) \right)\\&
 =& \lim_{n\to +\infty}\left(\;\int_{[{{ 0 }}, 1]} [u_n]^p_{s, p}\, d\mu^{+}(s) - \int_{[{{ 0 }}, 1]} [u_n -u]^p_{s, p}\, d\mu^{+}(s) \right).
\end{eqnarray*}
This establishes the desired result for~$\mu^+$.

We now focus on proving the claim in~\eqref{djieow34LLLtyb5o4yo3tu3493pp}
for~$\mu^-$.
For this, we use~\eqref{mu3} and Proposition~\ref{absorb} to see that
\begin{eqnarray*}&&
\int_{[{{ 0 }}, 1]} [u_n]_{s, p}^p \, d\mu^- (s)=
\int_{[{{ 0 }}, \overline s]} [u_n]_{s, p}^p \, d\mu^- (s) \le c_0(N,\Omega, p) \,\gamma\int_{[\overline s, 1]} [u_n]^p_{s, p} \, d\mu(s) \\&&\qquad\qquad
\le c_0(N,\Omega, p) \,\gamma\int_{[{{ 0 }}, 1]} [u_n]^p_{s, p}\, d\mu^+ (s).
\end{eqnarray*}
This and the fact that~$u_n$ is bounded in~${\mathcal{X}}_p(\Omega)$
give that the quantity
$$\int_{[{{ 0 }}, 1]} [u_n]_{s, p}^p \, d\mu^- (s)$$ is uniformly bounded in~$n$.

Therefore, one can repeat the same argument as above, replacing~$\mu^+$
with~$\mu^-$ and obtain the desired result in~\eqref{djieow34LLLtyb5o4yo3tu3493pp}.
\end{proof}

To check~$(\mathcal F_3)$, we also point out that:

\begin{proposition}\label{PSP}
Let $\theta_0\in(0,1)$ and
\begin{equation}\label{topolino}
c^* := \frac{s_\sharp}{N} \, \Big( (1-\theta_0)\mathcal S(p)\Big)^{N/s_\sharp p}.
\end{equation}

Then, there exists~$\gamma_0>0$, depending on~$N$, $\Omega$, $p$, $s_\sharp$ and~$\theta_0$, such that if~$\gamma\in[0,\gamma_0]$
and~$c\in(0,c^*)$, then
the functional in~\eqref{funp} satisfies the \PS{c} condition.
\end{proposition}

\begin{proof}
Take~$c\in(0, c^*)$ and let~$u_n \subset \mathcal{X}_p(\Omega)$ be a sequence verifying
\begin{equation}\label{PS1}\begin{split}\lim_{n\to+\infty}
E(u_n) =\,&\lim_{n\to+\infty} \frac1p [\rho_p(u_n)]^p -\frac1p\int_{[{{{0}}}, \overline s]} [u_n]_{s, p}^p \, d\mu^-(s) - \frac{\lambda}{p} \int_\Omega |u_n|^p \, dx - \frac{1}{p_{s_\sharp}^*} \int_\Omega |u_n|^{p_{s_\sharp}^*}\, dx\\
=\,& c\end{split}
\end{equation}
and
\begin{equation}\label{PS2}
\begin{split}&\lim_{n\to+\infty}
\sup_{v \in {\mathcal{X}}_p(\Omega)}
( dE (u_n), v) 
\\= \,&\lim_{n\to+\infty}
\sup_{v \in {\mathcal{X}}_p(\Omega)}\int_{[{{{0}}}, 1]}\left(\;
c_{N,s,p}\iint_{\R^{2N}} \frac{|u_n(x) - u_n(y)|^{p-2}\, (u_n(x) - u_n(y))\, (v(x) - v(y))}{|x - y|^{N+sp}}\, dx dy\right) d\mu^+(s)\\
&\qquad-\int_{[{{{0}}}, \overline s]}\left(\;c_{N,s,p}\iint_{\R^{2N}} \frac{|u_n(x) - u_n(y)|^{p-2}\, (u_n(x) - u_n(y))\, (v(x) - v(y))}{|x - y|^{N+sp}}\, dx dy\right) d\mu^-(s)\\
&\qquad- \lambda \int_\Omega |u_n|^{p-2}\, u_n v\, dx - \int_\Omega |u_n|^{p_{s_\sharp}^* - 2}\, u_n v\, dx \\=\,& 0.
\end{split}
\end{equation}
Testing~\eqref{PS2} with~$v:=-u_n$,
\begin{equation*}
\begin{split}0\le \,&\lim_{n\to+\infty}
\int_{[{{{0}}}, 1]}\left(\;c_{N,s,p}\iint_{\R^{2N}} \frac{|u_n(x) - u_n(y)|^{p}}{|x - y|^{N+sp}}\, dx dy\right) d\mu^+(s)\\
&\qquad-\int_{[{{{0}}}, \overline s]}\left(\;c_{N,s,p}\iint_{\R^{2N}} \frac{|u_n(x) - u_n(y)|^{p}}{|x - y|^{N+sp}}\, dx dy\right) d\mu^-(s)\\
&\qquad- \lambda \int_\Omega |u_n|^{p}\, dx - \int_\Omega |u_n|^{p_{s_\sharp}^* }\, dx\\
=\,&\lim_{n\to+\infty}
p\,E(u_n)+\left(\frac{p}{p^*_{s_\sharp}}-1
\right)\int_\Omega |u_n|^{p_{s_\sharp}^* }\, dx,
\end{split}
\end{equation*}
and then, by~\eqref{PS1}, 
\begin{equation}\label{boh}\lim_{n\to+\infty}
\left(\frac{1}{p}-\frac{1}{p^*_{s_\sharp}}\right) \|u_n\|_{L^{p^*_{s_\sharp}}(\R^N)}^{p^*_{s_\sharp}} \le c .
\end{equation}

Furthermore, from Proposition~\ref{absorb} we deduce that
\[
\int_{[{{{0}}}, \overline s]} [u_n]_{s, p}^p \, d\mu^-(s) \le c_0 \gamma \int_{[\overline s,1]} [u_n]_{ s, p}^p \, d\mu^+(s) \le c_0 \gamma[\rho_p(u_n)]^p .
\]
Hence,
\begin{equation}\label{PKSD-LPAIKSEASCMC}
[\rho_p(u_n)]^p -\int_{[{{{0}}}, \overline s]} [u_n]_{s, p}^p \, d\mu^-(s)\ge
\Big(1-C^{2p}(N,\Omega,p)\gamma\Big)
[\rho_p(u_n)]^p.
\end{equation}
{F}rom this and~\eqref{PS1}, we obtain that~$\rho_p(u_n)$ is bounded
uniformly in~$n$, provided that~$1-C^{2p}(N,\Omega,p)\gamma>0$.

Accordingly, in view of Lemma~\ref{Vitali}, there exists~$u\in\mathcal{X}_p(\Omega)$ such that, up to subsequences,
\begin{alignat}{2}\nonumber
& u_n\rightharpoonup u && \text{ in } \mathcal{X}_p(\Omega),\\\label{4.7}
& u_n\to u && \text{ in } L^r(\Omega) \text{ for any } r\in [1, p^*_{s_\sharp}),\\\nonumber
& u_n\to u && \text{ a.e.  in } \Omega.
\end{alignat}

It remains to prove that~$u_n \to u$ in~$\mathcal{X}_p(\Omega)$ as~$n\to+\infty$. 
For this, we set~$\widetilde{u}_n :=u_n - u$. By~\cite[Theorem~1]{MR699419} we have that
\begin{equation} \label{L32}
\|u\|^{p^*_{s_\sharp}}_{L^{p^*_{s_\sharp}}(\R^N)}
= \lim_{n\to+\infty}\|u_n\|^{p^*_{s_\sharp}}_{L^{p^*_{s_\sharp}}(\R^N)} - \|\widetilde{u}_n\|_{L^{p^*_{s_\sharp}}(\R^N)}^{p^*_{s_\sharp}}.
\end{equation}
Moreover, by~Lemma \ref{Brezis-Lieb} we get that
\begin{equation} \label{L32BIS}
\int_{[{{{0}}}, 1]} [u]^p_{s, p}\, d\mu^\pm(s)= 
\lim_{n\to+\infty}\int_{[{{{0}}}, 1]} [u_n]^p_{s, p}\, d\mu^\pm(s)-
\int_{[{{{0}}}, 1]} [\widetilde u_n]^p_{s, p}\, d\mu^\pm(s) 
.\end{equation}

Hence, testing Definition~\ref{wsol} with~$v:=u$, 
\begin{equation}\label{test1}
[\rho_p (u)]^p -\int_{[{{{0}}}, \overline s]} [u]^p_{s, p}\, d\mu^-(s) = \lambda \|u\|^p_{L^p(\R^N)} + \|u\|_{L^{p^*_{s_\sharp}}(\R^N)}^{p^*_{s_\sharp}}.\end{equation} 
Similarly, testing identity~\eqref{PS2} with~$v:=\pm u_n$,
\begin{equation*}
\lim_{n\to+\infty}
[\rho_p (u_n)]^p  -\int_{[{{{0}}}, \overline s]} [u_n]^p_{s, p}\, d\mu^-(s) - \lambda \|u_n\|^p_{L^p(\R^N)} - \|u_n\|_{L^{p^*_{s_\sharp}}(\R^N)}^{p^*_{s_\sharp}} =0.
\end{equation*}
This and~\eqref{4.7} give that
\begin{equation}\label{L32BIS03plrf-73874}
\lim_{n\to+\infty}
[\rho_p (u_n)]^p  -\int_{[{{{0}}}, \overline s]} [u_n]^p_{s, p}\, d\mu^-(s) - \|u_n\|_{L^{p^*_{s_\sharp}}(\R^N)}^{p^*_{s_\sharp}} =\lambda \|u\|^p_{L^p(\R^N)}.
\end{equation}
We compare this with~\eqref{test1} and we see that
\begin{equation*}
\lim_{n\to+\infty}
[\rho_p (u_n)]^p  -\int_{[{{{0}}}, \overline s]} [u_n]^p_{s, p}\, d\mu^-(s) - \|u_n\|_{L^{p^*_{s_\sharp}}(\R^N)}^{p^*_{s_\sharp}} =[\rho_p (u)]^p -\int_{[{{{0}}}, \overline s]} [u]^p_{s, p}\, d\mu^-(s)-\|u\|_{L^{p^*_{s_\sharp}}(\R^N)}^{p^*_{s_\sharp}}.
\end{equation*}
Hence, in light of~\eqref{L32BIS},
\begin{equation*}
\lim_{n\to+\infty}
[\rho_p (\widetilde u_n)]^p  -\int_{[{{{0}}}, \overline s]} [\widetilde u_n]^p_{s, p}\, d\mu^-(s) - \|u_n\|_{L^{p^*_{s_\sharp}}(\R^N)}^{p^*_{s_\sharp}} = -\|u\|_{L^{p^*_{s_\sharp}}(\R^N)}^{p^*_{s_\sharp}}.
\end{equation*}
Accordingly, by~\eqref{L32},
\begin{equation}\label{NBSFVCECGNAS}
\lim_{n\to+\infty}
[\rho_p (\widetilde u_n)]^p  -\int_{[{{{0}}}, \overline s]} [\widetilde u_n]^p_{s, p}\, d\mu^-(s) - \|\widetilde u_n\|_{L^{p^*_{s_\sharp}}(\R^N)}^{p^*_{s_\sharp}} = 0.
\end{equation}

Combining this and the definition of the Sobolev constant in~\eqref{DESOCO}, 
\begin{equation*}
\lim_{n\to+\infty}
[\rho_p (\widetilde u_n)]^p  -\int_{[{{{0}}}, \overline s]} [\widetilde u_n]^p_{s, p}\, d\mu^-(s) 
-\left(\frac1{{\mathcal S(p)}}\int_{[{{{0}}},1]} [\widetilde u_n]^p_{s, p}\, d\mu^+(s)\right)^{p^*_{s_\sharp}/p}
\le 0,
\end{equation*}
that is
\begin{equation*}
\lim_{n\to+\infty}
[\rho_p (\widetilde u_n)]^p  -\int_{[{{{0}}}, \overline s]} [\widetilde u_n]^p_{s, p}\, d\mu^-(s) 
-\frac{[\rho_p (\widetilde u_n)]^{p^*_{s_\sharp}}}{({\mathcal S(p)})^{p^*_{s_\sharp}/p}}
\le 0.
\end{equation*}
We now recall~\eqref{PKSD-LPAIKSEASCMC}, applied here to~$\widetilde u_n$ instead of~$u_n$, and we conclude that
\begin{equation*}
\lim_{n\to+\infty}
\Big(1-C^{2p}(N,\Omega,p)\gamma\Big)[\rho_p(\widetilde u_n)]^p
-\frac{[\rho_p (\widetilde u_n)]^{p^*_{s_\sharp}}}{({\mathcal S(p)})^{p^*_{s_\sharp}/p}}
\le 0,
\end{equation*}
or equivalently
\begin{equation}\label{moOmdRIJSmdfGOIANMSnao0lol}
\lim_{n\to+\infty}
[\rho_p(\widetilde u_n)]^p
\left[ \Big(1-C^{2p}(N,\Omega,p)\gamma\Big){({\mathcal S(p)})^{p^*_{s_\sharp}/p}}
-{[\rho_p (\widetilde u_n)]^{p^*_{s_\sharp}-p}}\right]\le 0.
\end{equation}

Now we reconsider~\eqref{PS1} in the light of~\eqref{L32BIS03plrf-73874} and we see that
\begin{eqnarray*}c&=&\lim_{n\to+\infty}
\frac1p [\rho_p(u_n)]^p -\frac1p\int_{[{{{0}}}, \overline s]} [u_n]_{s, p}^p \, d\mu^-(s) - \frac{\lambda}{p} \int_\Omega |u_n|^p \, dx - \frac{1}{p_{s_\sharp}^*} \int_\Omega |u_n|^{p_{s_\sharp}^*}\, dx\\&=&\lim_{n\to+\infty}
\left(\frac1p-\frac1{p^*_{s_\sharp}}\right)\left( [\rho_p(u_n)]^p -\int_{[{{{0}}}, \overline s]} [u_n]_{s, p}^p \, d\mu^-(s) \right)- \frac{\lambda}{p} \int_\Omega |u_n|^p \, dx +
\frac\lambda{p^*_{s_\sharp}}\int_\Omega|u|^p\,dx.
\end{eqnarray*}
This and the strong convergence in~\eqref{4.7} yield that
\begin{eqnarray*}c&=&\lim_{n\to+\infty}
\frac{s_\sharp}N\left( [\rho_p(u_n)]^p -\int_{[{{{0}}}, \overline s]} [u_n]_{s, p}^p \, d\mu^-(s) \right)- \frac{\lambda \,s_\sharp}{N} \int_\Omega |u|^p \, dx.
\end{eqnarray*}
{F}rom this and~\eqref{L32BIS} we arrive at
\begin{eqnarray*}\frac{cN}{s_\sharp}&=&\lim_{n\to+\infty}
[\rho_p(u)]^p+[\rho_p(\widetilde u_n)]^p -\int_{[{{{0}}}, \overline s]} [u]_{s, p}^p \, d\mu^-(s)-\int_{[{{{0}}}, \overline s]} [\widetilde u_n]_{s, p}^p \, d\mu^-(s)- \lambda\int_\Omega |u|^p \, dx.
\end{eqnarray*}
As a consequence, by~\eqref{test1},
\begin{eqnarray*}\frac{cN}{s_\sharp}&=&\lim_{n\to+\infty}
[\rho_p(\widetilde u_n)]^p -\int_{[{{{0}}}, \overline s]} [\widetilde u_n]_{s, p}^p \, d\mu^-(s)
+\int_\Omega |u|^{p^*_{s_\sharp}} \, dx.
\end{eqnarray*}
Hence, recalling~\eqref{PKSD-LPAIKSEASCMC} (used here for~$\widetilde u_n$ instead of~$u_n$),
\begin{eqnarray*}\frac{cN}{s_\sharp}&\ge&\lim_{n\to+\infty}\Big(1-C^{2p}(N,\Omega,p)\gamma\Big)
[\rho_p(\widetilde u_n)]^p +\int_\Omega |u|^{p^*_{s_\sharp}} \, dx\\
&\ge&\lim_{n\to+\infty}\Big(1-C^{2p}(N,\Omega,p)\gamma\Big)
[\rho_p(\widetilde u_n)]^p.
\end{eqnarray*}
This and~\eqref{topolino} tell us that
\begin{eqnarray*}\Big( (1-\theta_0)\mathcal S(p)\Big)^{N/s_\sharp p}=
\frac{c^* N}{s_\sharp}>
\frac{cN}{s_\sharp}\ge\lim_{n\to+\infty}\Big(1-C^{2p}(N,\Omega,p)\gamma\Big)
[\rho_p(\widetilde u_n)]^p
\end{eqnarray*}
and consequently
\begin{eqnarray*} &&\liminf_{n\to+\infty}\Big(1-C^{2p}(N,\Omega,p)\gamma\Big){({\mathcal S(p)})^{p^*_{s_\sharp}/p}}
-{[\rho_p (\widetilde u_n)]^{p^*_{s_\sharp}-p}}\\&&\qquad=\Big(1-C^{2p}(N,\Omega,p)\gamma\Big){({\mathcal S(p)})^{N/(N-s_\sharp p)}}-\limsup_{n\to+\infty} [\rho_p (\widetilde u_n)]^{{s_\sharp}p^2/(N-{s_\sharp}p)}\\
&&\qquad\ge
\Big(1-C^{2p}(N,\Omega,p)\gamma\Big){({\mathcal S(p)})^{N/(N-s_\sharp p)}}-
\left(\frac{\Big( (1-\theta_0)\mathcal S(p)\Big)^{N/s_\sharp p}}{1-C^{2p}(N,\Omega,p)\gamma
}\right)^{{s_\sharp}p/(N-{s_\sharp}p)}\\&&\qquad=
\left[
1-C^{2p}(N,\Omega,p)\gamma-
\left(\frac{ (1-\theta_0)^{N/s_\sharp p}}{1-C^{2p}(N,\Omega,p)\gamma
}\right)^{{s_\sharp}p/(N-{s_\sharp}p)}\right]
{({\mathcal S(p)})^{N/(N-s_\sharp p)}}.
\end{eqnarray*}

Thus, since
$$\lim_{\gamma\searrow0}
1-C^{2p}(N,\Omega,p)\gamma-
\left(\frac{ (1-\theta_0)^{N/s_\sharp p}}{1-C^{2p}(N,\Omega,p)\gamma
}\right)^{{s_\sharp}p/(N-{s_\sharp}p)}=
1-(1-\theta_0)^{N/(N-{s_\sharp}p)}>0,$$
we infer that
$$ \liminf_{n\to+\infty}\Big(1-C^{2p}(N,\Omega,p)\gamma\Big){({\mathcal S(p)})^{p^*_{s_\sharp}/p}}
-{[\rho_p (\widetilde u_n)]^{p^*_{s_\sharp}-p}}>0,$$
as long as~$\gamma$ is sufficiently small.

Combining this information with~\eqref{moOmdRIJSmdfGOIANMSnao0lol}, we gather that
$$ \lim_{n\to+\infty} \rho_p(\widetilde u_n)=0,$$
and thus $u_n \to u$ in~$\mathcal{X}_p(\Omega)$ as~$n\to+\infty$.
\end{proof}

Moreover, as it concerns~$(\mathcal N_1)$, in view of~\eqref{Npippoepluto}
and Theorem~\ref{SOBOLEVEMBEDDING} we have that
$$ N_p(u)=\frac1p\int_{[{{{0}}}, \overline s]} [u]_{s, p}^p \, d\mu^-(s)\le 
\frac{C(N,\Omega,p)\,\mu^-\big( {[{{{0}}}, \overline s]}\big)}p\,[u]^p_{\overline s,p}.$$
Hence, recalling~\eqref{mu2} and using again Theorem~\ref{SOBOLEVEMBEDDING},
$$ N_p(u)\le\frac{C(N,\Omega,p)\,\gamma\,\mu^+\big([\overline s, 1]\big)}p\,[u]^p_{\overline s,p}
\le\frac{\big( C(N,\Omega,p)\big)^2\,\gamma}p\,\int_{[\overline s, 1]}[u]^p_{ s,p}\,d\mu^+(s)
.$$
As a result, we can take
$$ \eta:=\big( C(N,\Omega,p)\big)^2\,\gamma $$
and notice that, if~$\gamma$ is sufficiently small, possibly in dependence
of~$N$, $\Omega$ and~$p$, then~$\eta\in(0,1)$.
In this case,
$$N_p(u)\le
\frac{\eta}p \,[\rho_p(u)]^p=\eta \,I_p(u),
$$
and~$(\mathcal N_1)$ is thereby established.
\medskip

In this way we have:

\begin{proof}[Proof of Theorem~\ref{mainp1}]
We have just checked that all the abstract assumptions in Section~\ref{AB2} are fulfilled
under the hypotheses of Theorem~\ref{mainp1}. We can thereby exploit Theorem~\ref{Theorem 2.9}, with the choices of~$q$, $\beta$ and~$c^* $
as in~\eqref{plutoepaper}
and~\eqref{topolino},
which in turn implies the thesis of Theorem~\ref{mainp1} as a particular case.

To make this argument work, we have to check that one can choose~$\theta_0\in(0,1)$
in~\eqref{topolino} so to satisfy~\eqref{2.11}. To this end, we observe that~\eqref{2.11} in this specific case boils down to
\begin{equation*} \lambda_l - \left(\frac1{|\Omega|^{(p_{s_\sharp}^* -p)/p }}\right)^{p/p^*_{s_\sharp}} \left(\frac{p\,p^*_{s_\sharp}\,c^*}{p^*_{s_\sharp} - p}\right)^{1 - p/p^*_{s_\sharp}} < \lambda < \lambda_l ,\end{equation*}
that is
\begin{equation*} \lambda_l - \left(\frac{N\,c^*}{|\Omega|\, {s_\sharp} }\right)^{{s_\sharp}p/N} < \lambda < \lambda_l ,\end{equation*}
which in turn, recalling~\eqref{topolino}, reduces to
\begin{equation*} \lambda_l - \frac{(1-\theta_0)\mathcal S(p)}{|\Omega|^{{s_\sharp}p/N}} < \lambda < \lambda_l .\end{equation*}
This condition is fulfilled thanks to the structural assumption in~\eqref{LLP1}, which allows us to define
\begin{equation*} \theta_0:=
\frac12\left(1-\frac{|\Omega|^{{s_\sharp}p/N}}{\mathcal S(p)}(\lambda_l-\lambda)\right)\in(0,1).\qedhere\end{equation*}
\end{proof}

\section{Applications}\label{sec-app}

The overall versatility of the measure $\mu$ that we treat in this paper allows us
to treat a variety of operators.
In this section we focus our attention on illustrating some examples that stem from the particular choice of the measure~$\mu$.

We point out that most of the multiplicity results provided here are new and they enrich the already existing literature.
\medskip

To start with, we showcase multiplicity results for the classical~$p$-Laplacian and the fractional~$p$-Laplacian. Indeed, with specific choices of the measure~$\mu$,
Theorem~\ref{mainp1} recovers~\cite[Theorem 1.1]{MR3469053} and~\cite[Theorem 1.1]{PSY}).

\begin{corollary}\label{C1}
Let~$p\in (1, N)$, let~$p^*:= Np/(N-p)$ be the classical critical Sobolev exponent and $S_p$ the classical best Sobolev constant.  Suppose that, for some~$l$, $ m \ge 1$, 
\begin{equation*}
\lambda_l - \frac{S_{p}}{|\Omega|^{p/N}} < \lambda < \lambda_l = \dots = \lambda_{l+m-1}.
\end{equation*}

Then, the following problem
\begin{equation*}
\left\{\begin{aligned}
-\Delta_p \, u & = \lambda |u|^{p-2} u + |u|^{p^* - 2}\, u && \text{in } \Omega,\\
u & = 0 && \text{on } \partial\Omega\end{aligned}\right.
\end{equation*}
has~$m$ distinct pairs of nontrivial solutions~$\pm u^\lambda_1,\dots,\pm u^\lambda_m$.
\end{corollary}

\begin{proof}
Let~$\mu:= \delta_1$ be the Dirac measure centered at the point~$1$.
Notice that~$\mu$ verifies conditions~\eqref{mu00}, \eqref{mu3}
and~\eqref{mu2} with~$\overline s:=1$ and~$\gamma=0$. 
Furthermore, we can take~$s_\sharp:=1$, so that~${\mathcal{S}(p)}$ defined in~\eqref{DESOCO} boils down to the classical Sobolev constant~$S_p$.
The desired result now follows from Theorem~\ref{mainp1}.
\end{proof}

\begin{corollary}\label{C2}
Let~$s\in ({{ 0 }}, 1)$ and~$p\in (1, N)$. 
Let~$p_s^* := Np/(N-sp)$ be the fractional critical Sobolev exponent and let~$S_p(s)$ be the fractional best Sobolev constant.  Suppose that, for some~$l$, $m \ge 1$,
\begin{equation}\label{lambda-pslap}
\lambda_l - \frac{S_{p}(s)}{|\Omega|^{sp/N}} < \lambda < \lambda_l = \dots = \lambda_{l+m-1}.
\end{equation}

Then, the following problem
\begin{equation*}
\left\{\begin{aligned}
(-\Delta)^s_p \, u & = \lambda |u|^{p-2} u + |u|^{p_s^* - 2}\, u && \text{in } \Omega,\\
u & = 0 && \text{in } \R^N \setminus \Omega\end{aligned}\right.
\end{equation*}
has~$m$ distinct pairs of nontrivial solutions~$\pm u^\lambda_1,\dots,\pm u^\lambda_m$.
\end{corollary}

\begin{proof}
In this case we take~$\mu:= \delta_s$, being~$\delta_s$
the Dirac measure centered at~$s$. 
With this choice, $\mu$ satisfies the conditions~\eqref{mu00}, \eqref{mu3}
and~\eqref{mu2} with~$\overline s:=s$ and~$\gamma=0$.
Moreover, here we can take~$s_\sharp:=s$, and then the desired result follows from Theorem~\ref{mainp1}.
\end{proof}

The next result provides the existence of multiple solutions for the mixed operator $-\Delta_p + (-\Delta)_p^s$. To our best knowledge, this result is new even for the case $p=2$:
\begin{corollary}\label{C3}
Let~$s\in [{{ 0 }}, 1)$ and~$p\in (1, N)$. 
Let~$p^* = Np/(N-p)$ be the classical critical Sobolev exponent and let $S_p$ be the best Sobolev constant.  Suppose that, for some~$l$, $ m \ge 1$,
\begin{equation*}
\lambda_l - \frac{S_{p}}{|\Omega|^{p/N}} < \lambda < \lambda_l = \dots = \lambda_{l+m-1}.
\end{equation*}

Then, the following problem
\begin{equation*}
\left\{\begin{aligned}
-\Delta_p \, u + (-\Delta)_p^s \, u & = \lambda |u|^{p-2} u + |u|^{p^* - 2}\, u && \text{in } \Omega,\\
u & = 0 && \text{in } \R^N \setminus \Omega\end{aligned}\right.
\end{equation*}
has~$m$ distinct pairs of nontrivial solutions~$\pm u^\lambda_1,\dots,\pm u^\lambda_m$.
\end{corollary}

\begin{proof}
Let~$\delta_1$ and~$\delta_s$ denote the Dirac measures centered at~$1$ and~$s$, respectively. We define~$\mu:=\delta_1 + \delta_s$.  Here we take~$\overline s:=1$, $\gamma=0$ and~$s_\sharp:=1$
and we observe that conditions~\eqref{mu00}, \eqref{mu3}
and~\eqref{mu2} are fulfilled.
Thus, the desired result follows from Theorem~\ref{mainp1}.
\end{proof}

%
%

An interesting scenario turns out when the measure $\mu$ changes sign. It means, for example, that the operator is allowed to include a small term with the ``wrong'' sign. Again, to our best knowledge, the existing literature lacks of a result of this kind.

\begin{corollary}\label{C5}
Let~$s\in [{{ 0 }}, 1)$, ~$p\in (1, N)$ and~$\alpha\in\R$.
Denote by~$p^* := Np/(N-p)$ the classical critical Sobolev exponent and by~$S_p$ the classical Sobolev constant.
Suppose that, for some~$l$, $m \ge 1$,
\begin{equation*}
\lambda_l - \frac{S_{p}}{|\Omega|^{p/N}} < \lambda < \lambda_l = \dots = \lambda_{l+m-1}
.\end{equation*} 

Then, there exists~$\alpha_0>0$, depending only on~$N$, $\Omega$, $p$,
$s$, $\lambda$ and~$l$ such that if~$\alpha\le\alpha_0$,
then the following problem 
\begin{equation*}
\left\{\begin{aligned}
-\Delta_p \, u - \alpha(-\Delta)_p^s \, u & = \lambda |u|^{p-2} + |u|^{p^* - 2}\, u && \text{in } \Omega,\\
u & = 0 && \text{in } \R^N \setminus \Omega
\end{aligned}\right.
\end{equation*}
has~$m$ distinct pairs of nontrivial solutions~$\pm u^\lambda_1,\dots,\pm u^\lambda_m$.
\end{corollary}

\begin{proof}
We define~$\mu:= \delta_1 -\alpha\delta_s$ and take~$\overline s:=1$
and~$s_\sharp:=1$. In this way,
conditions~\eqref{mu00} and~\eqref{mu3} are satisfied. Furthermore,
\begin{eqnarray*}
\mu^-\big([{{ 0 }}, \overline s]\big)\le \max\{0,\alpha\}
= \max\{0,\alpha\}\mu^+\big([\overline s, 1]\big),
\end{eqnarray*}
which gives that condition~\eqref{mu2} holds true taking~$\gamma:=\max\{0,\alpha\}$.

Hence, we obtain the desired result as an application of
Theorem~\ref{mainp1}.  
\end{proof}

Our setting allows us to choose~$\mu$ as a convergent series of Dirac measures. To be more precise, the next two results hold true.

\begin{corollary}\label{serie1}
Let ~$p\in (1, N)$ and~$1\ge s_0 > s_1> s_2 >\dots \ge 0$. Consider the operator
\[
\sum_{k=0}^{+\infty} c_k (-\Delta)_p^{s_k} 
\qquad{\mbox{with }} \,\sum_{k=0}^{+\infty} c_k \in (0, +\infty).
\]
Suppose that~$ c_0>0$ and~$c_k\ge 0$ for all $k\ge 1$.

Denote by~$p^*_{s_0} := Np/(N-ps_0)$ the fractional critical Sobolev exponent and by~$S_p(s_0)$ the fractional Sobolev constant corresponding to the exponent $s_0$. Assume that, for some~$l$, $m \ge 1$,
\begin{equation*}
\lambda_l - \frac{S_{p}(s_0)}{|\Omega|^{s_0 p/N}} < \lambda < \lambda_l = \dots = \lambda_{l+m-1}.
\end{equation*}

Then, the following problem
\begin{equation}\label{serie}
\left\{\begin{aligned}
\sum_{k=0}^{+\infty} c_k (-\Delta)_p^{s_k} u & = \lambda |u|^{p-2} u + |u|^{p_{s_0}^* - 2}\, u && \text{in } \Omega,\\
u & = 0 && \text{in } \R^N \setminus \Omega
\end{aligned}\right.
\end{equation}
has~$m$ distinct pairs of nontrivial solutions~$\pm u^\lambda_1,\dots,\pm u^\lambda_m$.
\end{corollary}

\begin{proof}
Here we set
\[
\mu:=\sum_{k=0}^{+\infty} c_k \,\delta_{s_k} , 
\]
where~$\delta_{s_k}$ denote the Dirac measures centered at each~$s_k$. In this case,
we can take~$\overline s:=s_0$ and~$s_\sharp:=s_0$,
and notice that conditions~\eqref{mu00}, \eqref{mu3}
and~\eqref{mu2} are satisfied (with~$\gamma=0$).

Hence, the desired result is a byproduct of Theorem~\ref{mainp1}.
\end{proof}

\begin{corollary}\label{serie2}
Let ~$p\in (1, N)$ and~$1\ge s_0 > s_1> s_2 >\dots \ge 0$.  
Consider the operator
\[
\sum_{k=0}^{+\infty} c_k (-\Delta)_p^{s_k} 
\qquad{\mbox{with }} \,\sum_{k=0}^{+\infty} c_k \in (0, +\infty).
\]
Assume that there exist~$\gamma\ge0$ and~$\overline k\in\N$ such that 
\begin{equation}\label{<gamma}
c_k>0\ \text{ for all } k\in\{0,\dots, \overline k\}\quad \text{ and }\quad \sum_{k=\overline k +1}^{+\infty} c_k \le \gamma \sum_{k=0}^{\overline k} c_k.
\end{equation}

Denote by~$p^*_{s_0} := Np/(N-ps_0)$ the fractional critical Sobolev exponent and by~$S_p(s_0)$ the fractional Sobolev constant corresponding to the exponent~${s_0}$. Assume that, for some~$l$, $m \ge 1$,
\begin{equation*}
\lambda_l - \frac{S_{p}(s_0)}{|\Omega|^{s_0 p/N}} < \lambda < \lambda_l = \dots = \lambda_{l+m-1}
\end{equation*}
holds.

Then, there exists~$\gamma_0>0$ depending on~$N$, $\Omega$, $p$, $s_k$, $c_k$, $\lambda$ and~$l$ such that
if~$\gamma\in[0,\gamma_0]$, problem~\eqref{serie} has~$m$ distinct pairs of nontrivial solutions~$\pm u^\lambda_1,\dots,\pm u^\lambda_m$.\end{corollary}

\begin{proof}
We set once again
\[
\mu:=\sum_{k=0}^{+\infty} c_k \,\delta_{s_k}
\]
where~$\delta_{s_k}$ denote the Dirac measures centered at each~$s_k$. 

In this case, condition~\eqref{<gamma} assures that the assumptions~\eqref{mu00},
\eqref{mu3} and~\eqref{mu2} on the measure~$\mu$ are fulfilled
with~$\overline s:=s_{\overline k}$.
Thus, taking~$s_\sharp:=s_0$ we derive the desired result from
Theorem~\ref{mainp1}.
\end{proof}

We point out that another interesting result comes from the continuous superposition of fractional operators of $p$-Laplacian type.  To the best of our knowledge, also this result happens to be new:

\begin{corollary}\label{function}
Let~$s_\sharp\in ({{ 0 }}, 1)$, $\gamma\ge0$ and~$f$ be a measurable and non identically zero function satisfying
\begin{equation}\label{fun-int}\begin{split} &
{\mbox{$f\ge 0$ in~$(s_\sharp,1)$,}}\\&
\int_{s_\sharp}^1 f(s) \,ds >0
\\ {\mbox{and }}\qquad&
\int_0^{s_\sharp} \max\{0,-f(s)\} \,ds \le\gamma \int_{s_\sharp}^1 f(s) \,ds .
\end{split}\end{equation}

Denote by~$p_{s_\sharp}^* := Np/(N-ps_\sharp)$ the fractional critical Sobolev exponent
and by~$S_p(s_\sharp)$ the best Sobolev constant corresponding to
the exponent~$s_\sharp$.  Assume that, for some~$l$, $m \ge 1$,
\begin{equation*}
\lambda_l - \frac{S_{p}(s_\sharp)}{|\Omega|^{s_\sharp p/N}} < \lambda < \lambda_l = \dots = \lambda_{l+m-1}.
\end{equation*}

Then, there exists~$\gamma_0>0$, depending only on~$N$, $\Omega$, $p$, $s_\sharp$, $f$, $\lambda$ and~$l$, such that if~$\gamma\in[0,\gamma_0]$,
then problem
\begin{equation*}
\left\{\begin{aligned}
\int_0^1 f(s) (-\Delta)_p^s \, u \, ds & = \lambda |u|^{p-2} u + |u|^{p_{s_\sharp}^* - 2}\, u && \text{in } \Omega,\\
u & = 0 && \text{in } \R^N \setminus \Omega,
\end{aligned}\right.
\end{equation*}
has~$m$ distinct pairs of nontrivial solutions~$\pm u^\lambda_1,\dots,\pm u^\lambda_m$. \end{corollary}

\begin{proof}
In this case, the operator~$A_{\mu, p}$ takes the form
\[
\int_0^1 f(s) (-\Delta)_p^s \, u \, ds,
\]
which means that we are taking~$d\mu(s) = f(s)\,ds$.

Moreover, here $s_\sharp$ acts as the critical fractional Sobolev exponent.
Owing to the conditions stated in \eqref{fun-int}, we have that~\eqref{mu00}, \eqref{mu3} and~\eqref{mu2} are fulfilled with~$\overline s:=s_\sharp$.
Hence, the desired result follows from Theorem~\ref{mainp1}.
\end{proof}

\section*{Acknowledgements} 

SD and EV are members of the Australian Mathematical Society (AustMS).
EV is supported by the Australian Laureate Fellowship FL190100081 ``Minimal surfaces, free boundaries and partial differential equations''.

CS is member of INdAM-GNAMPA.

This work was partially completed while KP was visiting the Department of Mathematics and Statistics at the University of Western Australia, and he is grateful for the hospitality of the host department. His visit to the UWA was supported by the Simons Foundation Award 962241 ``Local and nonlocal variational problems with lack of compactness''.

\vfill

\end{document}